\documentclass[12pt]{amsart}

\usepackage[a4paper,top=2.5cm,bottom=3.0cm,left=2.6cm,right=2.4cm]{geometry} 

\usepackage{amssymb}
\usepackage{graphicx}
\usepackage{times}
\usepackage{mathrsfs}
\usepackage{rsfso}
\usepackage{enumerate}
\usepackage{color}
\usepackage[colorlinks,citecolor=blue,linkcolor=red,
            bookmarksopen,
            bookmarksnumbered
           ]{hyperref}

\newtheorem{theorem}{Theorem}[section]
\newtheorem{lemma}[theorem]{Lemma}
\newtheorem{proposition}[theorem]{Proposition}
\newtheorem{corollary}[theorem]{Corollary}

\theoremstyle{definition}

\newtheorem{notation}[theorem]{Notation}

\theoremstyle{remark}
\newtheorem{remark}[theorem]{Remark}

\numberwithin{equation}{section}
\newcommand{\ee}{\hskip0.15ex}
\newcommand{\dd}[1]{_{\raise-0.3ex\hbox{$\scriptstyle #1$}}}
\newcommand{\sz}{\scriptstyle}
\newcommand{\un}{\underline}
\newcommand{\on}[1]{\raise-.5ex\hbox{\big|}_{#1}}

\newcommand {\Norm}[2]{ \mathchoice
    {|\ee #1\ee|\dd{#2}}
    {| #1 |_{#2}}
    {| #1 |_{#2}}
    {| #1 |_{#2}} }
\newcommand {\DNorm}[2]{ \mathchoice
    {\|\ee #1\ee\|\dd{#2}}
    {\| #1 \|_{#2}}
    {\| #1 \|_{#2}}
    {\| #1 \|_{#2}} }
\newcommand {\Normc}[2]{ \mathchoice
    {\left|\ee #1\ee\right|\dd{#2}^2}
    {| #1 |_{#2}^2}
    {| #1 |_{#2}^2}
    {| #1 |_{#2}^2} }
\newcommand {\DNormc}[2]{ \mathchoice
    {\left\|\ee #1\ee\right\|\dd{#2}^2}
    {\| #1 \|_{#2}^2}
    {\| #1 \|_{#2}^2}
    {\| #1 \|_{#2}^2} }

\newcommand{\loc}{\mathrm{loc}}
\newcommand{\cyl}{{\mathrm{cyl}}}

\newcommand{\strong}{{1,\ee\bullet}}

\newcommand\R{{\mathbb R}}
\newcommand\N{{\mathbb N}}
\renewcommand\P{{\mathbb P}}

\newcommand\T{{\mathbb T}}
\newcommand\Z{{\mathbb Z}}
\newcommand{\Id}{\mathbb I}

\newcommand\cA{{\mathcal{A}}}
\newcommand\cD{{\mathcal{D}}}
\newcommand\cF{{\mathcal{F}}}
\newcommand\cG{{\mathcal{G}}}
\newcommand\cI{{\mathcal{I}}}
\newcommand\cO{{\mathcal{O}}}
\newcommand\cR{{\mathcal{R}}}
\newcommand\cT{{\mathcal{T}}}

\newcommand{\rG}{{\mathcal{G}}}

\newcommand{\rd}{{\mathrm d}}

\newcommand{\ru}{\breve u}
\newcommand{\rv}{\breve v}
\newcommand{\rw}{\breve w}

\newcommand {\alphab}{{\boldsymbol{\alpha}}}
\newcommand {\betab}{{\boldsymbol{\beta}}}
\newcommand {\zetab}{{\boldsymbol{\zeta}}}

\newcommand {\be}{{\boldsymbol{\sf e}}}
\newcommand {\bu}{{\boldsymbol{u}}}
\newcommand {\bx}{{\boldsymbol{x}}}
\newcommand {\bz}{{\boldsymbol{z}}}

\newcommand {\bC}{{\mathbf{\un C}}}
\newcommand {\bH}{{\mathbf{\un H}}}

\begin{document}

\title[Characterization of Sobolev spaces by their Fourier coefficients]
{Characterization of Sobolev spaces by their Fourier coefficients in axisymmetric domains}

\author{Martin Costabel}
\author{Monique Dauge}
\author{Jun-Qi Hu}

\address{Univ. Rennes, CNRS, IRMAR - UMR 6625, F-35000 Rennes, France}

\email{Martin.Costabel@univ-rennes1.fr}
\email{Monique.Dauge@univ-rennes1.fr}

\address{Shanghai University of Finance and Economics, Shanghai, China}

\email{jqhu@fudan.edu.cn}

\thanks{The authors acknowledge support of the Centre Henri Lebesgue ANR-11-LABX-0020-01.
The third author thanks the China Scholarship Council for supporting this research by CSC visiting scholarship No.\ 201806485006.}



\begin{abstract}
Using Fourier series representations of functions on axisymmetric domains, we find weighted Sobolev norms of the Fourier coefficients of a function that yield norms equivalent to the standard Sobolev norms of the function.     
This characterization is universal in the sense that the equivalence constants are independent of the domain.
In particular it is uniform whether the domain contains a part of its axis of rotation or is disjoint from, but maybe arbitrarily close to, the axis. 
Our characterization using step-weighted norms involving the distance to the axis is  different from the one obtained earlier in the book [Bernardi, Dauge, Maday {\em Spectral methods for axisymmetric domains}, Gauthier-Villars, 1999], which involves trace conditions and is domain dependent. We also provide a complement for non cylindrical domains of the proof given in {\em loc.\ cit.} .
\end{abstract}
\maketitle



\section{Introduction}
\label{s:Intro}
\subsection{Motivation}
In $\R^{3}$, an axisymmetric domain $\breve\Omega$ is determined by its meridian domain $\Omega\subset\R_{+}\times\R$ via the set of its cylindrical coordinates 
\[
 \{(r,z,\theta) \mid (r,z)\in\Omega,\, \theta\in[0,2\pi]\}\,.
\]
Whether we include the axis $\{r=0\}$ in the domain  $\breve\Omega$ or not does not matter for the questions studied in this paper, see Remark~\ref{r:mazya1.1.18}. 

Fourier series in the angular variable can be used to reformulate certain 3D problems posed in $\breve\Omega$ as a sequence of 2D problems posed in $\Omega$. This is a standard technique (and we give a few selected typical references) for boundary value problems and spectral problems of mathematical physics and for their numerical approximation in the case where the physical problem is invariant by rotation, such as problems described by Laplace or wave equations \cite{Heinrich1996,HeinrichJung2006,Li2011,Li2014}, by the Lam\'e \cite{NkemziHeinrich1999,Chaussade-et-al2017}, Stokes or Navier-Stokes systems \cite{Abdellatif2000,Gallagher-et-al2001,LeeLi2012,CrowderErvin2013} or by Maxwell's equations \cite{AssousCiarletLabrunie2002,AssousCiarletLabrunie2003,AssousCiarletLabrunie2018,Oh2014,OH2015}. 
In this context, it is important to have a description of Sobolev norms on $\breve\Omega$ -- which can appear for example as energy functionals, but also as measures of approximation errors -- in terms of corresponding norms on $\Omega$ of the Fourier coefficients. As early as 1982, such a characterization of the Sobolev spaces of order $1,2,3$ has been given 
\cite{MercierRaugel1982}
as a tool for the analysis of a Fourier series \slash\ finite element approximation for second order elliptic Dirichlet problems. 

The standard reference for the characterization of $H^{s}(\breve\Omega)$ by Fourier coefficients for any positive real $s$ is the book \cite{BDMbook}, where this is a tool for the analysis of spectral methods. 
The present paper can be seen as a complement to the corresponding Chapter II of \cite{BDMbook} providing, for integer Sobolev indices $s$, an alternative approach to the question.
Sobolev spaces of integer order are, of course, simple because of the representation of the norm by $L^{2}$ norms of derivatives, but they are also limit cases for the trace mapping on the codimension $2$ manifold that is the axis of rotation of the 3D domain $\breve\Omega$. For the latter reason, in \cite{BDMbook} the results for integer order Sobolev spaces are obtained from those for non-integer orders by Hilbert space interpolation theory. Our present approach is more direct, using only the rewriting of partial derivatives in cylindrical coordinates and trying to simplify the expression of the resulting weighted Sobolev norms. 
The result is an equivalent norm where the equivalence constants do not depend on the domain, hence very general axisymmetric domains are allowed where the intersection with the axis of rotation is not necessarily an interval and the trace mapping is not well defined, or domains with a small hole around the axis, where answers can be found to the question of the behavior of the norms when the diameter of the hole tends to zero. This question cannot be answered by the methods of \cite{BDMbook}.

\subsection{Main results}   
Parseval's theorem for the Fourier series
\begin{equation}
\label{e:FS}
 u = \sum_{k\in\Z} u^{k}(r,z) \,e^{ik\theta}
\end{equation}
states that the mapping $u\mapsto (u^{k})_{k\in\Z}$ is unitary from the Hilbert space $L^{2}(\breve\Omega)$ to the direct sum of countably many copies of the Hilbert space $L^{2}_{1}(\Omega)$ of functions square integrable on the meridian domain $\Omega$ with respect to the natural measure $2\pi r\,\rd r\,\rd z$:
\[
 \DNormc{u}{L^{2}(\breve\Omega)} = \sum_{k\in\Z} \DNormc{u^{k}}{L^{2}_{1}(\Omega)}\,.
\]
For the Sobolev space $H^{m}(\breve\Omega)$, the norm of which is defined by
$
 \DNormc{u}{H^{m}(\breve\Omega)} = 
  \sum_{|\alpha|\le m} \DNormc{\partial^{\alpha}u}{L^{2}(\breve\Omega)}
$, a corresponding decomposition is possible:
\[
 \DNormc{u}{H^{m}(\breve\Omega)} = \sum_{k\in\Z} \DNormc{u^{k}}{H^{m}_{(k)}(\Omega)}\,.
\]
This is (almost) trivially true if the $H^{m}_{(k)}(\Omega)$ norms of the functions $u^{k}$ of $2$ variables on the right hand side are defined as the $H^{m}(\breve\Omega)$ norms of the functions of $3$ variables defined as $u^{k}e^{ik\theta}$ (see details in Notation \ref{n:dirsum} and Remark \ref{r:dirsum}).

The question is to describe the spaces $H^{m}_{(k)}(\Omega)$ and their norms more explicitly.
One ingredient to the answer is the Sobolev space $H^{m}_{1}(\Omega)$ defined as the subspace of $L^{2}_{1}(\Omega)$ of functions with finite norm
\[
 \DNormc{v}{H^{m}_{1}(\Omega)} = \sum_{|\alpha|\le m} 
  \DNormc{\partial^{\alpha}_{(r,z)}v}{L^{2}_{1}(\Omega)} \,.
\]
The answer given in \cite[Th. II.3.1]{BDMbook} is to describe $H^{m}_{(k)}(\Omega)$ as a subspace of $H^{m}_{1}(\Omega)$ defined by the vanishing of certain traces on the axis $\{r=0\}$. In some cases (when $|k|>m-1$ or $m-k$ is not an even integer), it is a closed subspace with norm equivalent to the norm of $H^{m}_{1}(\Omega)$, in the other cases it is a non-closed subspace where the vanishing trace condition for derivatives of order $m-1$ in $r$ is replaced by a finiteness condition of some weighted $L^{2}$ norm. The proof of this result in \cite{BDMbook} is given in detail only for the case of a cylindrical axisymmetric domain, i.e. when $\Omega$ is a rectangle.
We give a concise formulation of the result in Section~\ref{s:PF} and discuss its validity for polygonal $\Omega$ in the Appendix.

The main result of this paper is an equivalent characterization of the spaces $H^{m}_{(k)}(\Omega)$.

\begin{theorem}
\label{t:main}     
 For any $m\in\N$ there exist positive constants $c_{m}$ and $C_{m}$ such that for any $k\in\Z$ and any meridian domain $\Omega\subset\R_{+}\times\R$ there holds the norm equivalence
\begin{equation}
\label{e:normeq}
\begin{aligned}
c_{m} \DNormc{u^{k}}{H^{m}_{(k)}(\Omega)} &\le
  \!\sum_{\ell=0}^{\min\{|k|,m\}} 
  \DNormc{\left(\tfrac{|k|}{r}\right)^{\ell} {u^k} }{H^{m-\ell}_1(\Omega)}
  \!\!+\!\!
 \sum_{\ell=1}^{[(m-|k|)/2]}
   \DNormc{\left(\tfrac{1}{r}\partial_{r}\right)^{\ell}
      \left(\tfrac{1}{r}\right)^{|k|} {u^k} }{H^{m-|k|-2\ell}_1(\Omega)}
      \hskip-2ex\\
  &\le C_{m} \DNormc{u^{k}}{H^{m}_{(k)}(\Omega)} 
\end{aligned}
\end{equation}
Here the second sum extends over all $\ell\in\N$ satisfying 
$1\le\ell\le(m-|k|)/2$.

The same results holds in the same form if $\Omega$ is an interval of $\R_+$ (so $\breve\Omega$ is a disc or an annulus in $\R^{2}$), with the natural definition of the $H^m_1(\Omega)$-norm.
\end{theorem}

For further reference we introduce the notation $\DNorm{w}{C^m_{(k)}(\Omega)}$ for the norm defined by the sum of weighted norms appearing in \eqref{e:normeq}:
\begin{equation}
\label{e:Cmk}
  \DNormc{w}{C^m_{(k)}(\Omega)} = \sum_{\ell=0}^{\min\{|k|,m\}} 
   \DNormc{\left(\tfrac{|k|}{r}\right)^{\ell} w }{H^{m-\ell}_1(\Omega)}
   +
   \sum_{\ell=1}^{[(m-|k|)/2]}
   \DNormc{\left(\tfrac{1}{r}\partial_{r}\right)^{\ell}
   \left(\tfrac{1}{r}\right)^{|k|} w }{H^{m-|k|-2\ell}_1(\Omega)}.
\end{equation}

An easy consequence of the theorem is that for a given $m\in\N$, whereas for $|k|\ge m$ all spaces $H^{m}_{(k)}(\Omega)$ are the same, with equivalent norms (non-uniformly in $k$), the spaces for $0\le k \le m-1$ are all different, in general.

A similar description can be given for the Fourier coefficients of the radial and angular components of a vector valued $H^{m}(\breve\Omega)$ function, see Theorem~\ref{t:mainV}.

In Sections \ref{s:FPO}-\ref{s:Main} we fill in details of the definitions, formulate the result in more detail in terms of seminorms, and give a proof. 

In Section \ref{s:PF}, we compare the result of Theorem \ref{t:main} with the characterization given in \cite[Chap. II]{BDMbook}: Some examples are discussed involving functions that depend polynomially on the radial variable $r$. The situation of the limit as $\varepsilon\to0$ of a domain with a small hole $\{r\le \varepsilon\}$ can be handled explicitly.

In Section \ref{s:vect} we address the case of radial and angular components of a vector field.


\section{Fourier projection operators in axisymmetric domains}
\label{s:FPO}
\subsection{Fourier projection operators}
Denote by $\T=\R/(2\pi\Z)$ the standard one-dimensional torus and $\bx=(x,y,z)$ Cartesian variables in $\R^3$.
An axisymmetric domain $\breve\Omega$ is a domain in $\R^3$ that is invariant by rotation around some axis $\cA$, say the $z$-axis. A good way to introduce axisymmetry and related notions is to consider the group of rotations around the axis $\cA$: For all $\theta\in\T=\R/2\pi\Z$, let $\cR_\theta$ be the rotation of angle $\theta$ around $\cA$. So we assume
\[
   \forall\theta\in\T,\quad \cR_\theta\breve\Omega = \breve\Omega.
\]
Let $\ru$ be any scalar function in $L^2(\breve\Omega)$.
If we define the transformation $\rG_\theta$ by $(\rG_\theta \ru)(\bx)=\ru(\cR_\theta\bx)$, we obtain that the set of transformations $\big(\rG_\theta\big)_{\theta\in\T}$ has a group structure, isomorphic to that of the torus $\T$:
\[
   \rG_\theta \circ \rG_{\theta'} = \rG_{\theta+\theta'},\quad \theta,\theta'\in\T.
\]
Then, for each integer $k\in\Z$ we introduce the following angular Fourier transformation operator $\cF^k:\ru\mapsto\cF^k\ru$ defined from $L^2(\breve\Omega)$ into itself by
\begin{equation}
\label{eq:four}
   (\cF^k\ru)(\bx) = 
   \frac{1}{2\pi} \int_\T (\rG_\theta \ru)(\bx)\,e^{-ik\theta}\,\rd\theta,
   \quad \bx\in\breve\Omega,\quad k\in\Z.
\end{equation}

We call $\cF^k\ru$ 
the $k$-th \textit{Fourier component} of $\ru$. Note that it is  
still defined on the whole 3-dimensional domain $\breve\Omega$. 
An immediate consequence of its definition \eqref{eq:four} is that it satisfies
\begin{equation}
\label{eq:inva}
   \rG_\theta (\cF^k\ru) = e^{ik\theta}\,\cF^k\ru,\quad  \theta\in\T\,.
\end{equation}
It is, of course, completely defined by the \textit{Fourier coefficient} $u^{k}$, which is a function of two variables defined on the meridian domain, notions that will be introduced and studied in Section~\ref{s:FC}.

An obvious consequence of \eqref{eq:inva} is that the $\cF^k$ are projection operators such that (here $\delta_{kk'}$ is the Kronecker symbol):
\begin{equation}
\label{eq:proj}
   \cF^k\circ\cF^{k'} = \delta_{kk'}\cF^k,\quad k,k'\in\Z.
\end{equation}
For $\ru$ and $\rv$ in $L^2(\breve\Omega)$, we note that for each $\theta\in\T$
\[
\begin{aligned}
   \int_{\breve\Omega} (\rG_\theta \ru)(\bx)\,e^{-ik\theta} \;\overline{\rv}(\bx)\,\rd\bx &=
   \int_{\breve\Omega} \ru(\bx)\; e^{-ik\theta} \,(\rG_{-\theta} \overline{\rv})(\bx)\,\rd\bx \\ &=
   \int_{\breve\Omega} \ru(\bx)\; \overline{e^{-ik(-\theta)} \,(\rG_{-\theta} \rv)}(\bx)\,\rd\bx,
\end{aligned}
\]
and integrating for $\theta\in\T$ we obtain that $\cF^k$ is Hermitian
\begin{equation}
\label{eq:herm}
   \int_{\breve\Omega} (\cF^k \ru)(\bx) \;\overline{\rv}(\bx)\,\rd\bx =
   \int_{\breve\Omega} \ru(\bx) \;\overline{\cF^k\rv}(\bx)\,\rd\bx\,.
\end{equation}

We notice that for a.e.\ $\bx_0\in\breve\Omega$, $(\cF^k \ru)(\bx_0)$ is the $k$-th Fourier coefficient of the periodic function
\[
   \theta \mapsto (\rG_\theta u)(\bx_0)
\]
defined in $L^2(\T)$. Therefore we find for a.e.\ $\bx_0\in\breve\Omega$
\[
   \sum_{k\in\Z} e^{ik\theta} (\cF^k \ru)(\bx_0) = (\rG_\theta \ru)(\bx_0),\quad\theta\in\T
\]
But, by \eqref{eq:inva}, $e^{ik\theta} (\cF^k \ru) = \cG_\theta (\cF^k \ru)$. Hence
\[
   \sum_{k\in\Z} \cG_\theta (\cF^k \ru)(\bx_0) = (\rG_\theta \ru)(\bx_0),\quad\theta\in\T,
\]
which yields
\[
   \sum_{k\in\Z}  (\cF^k \ru)(\bx_0) =  \ru(\bx_0),\quad\bx_0\in\breve\Omega.
\]
We have obtained

\begin{proposition}
\label{pr:1}
The family of operators $\big(\cF^k\big)_{k\in\Z}$ \eqref{eq:four} defines a series of orthogonal projectors in $L^2(\breve\Omega)$ that satisfies
\[
   \sum_{k\in\Z} \cF^k = \Id.
\]
There is a decomposition of the norm in $L^2(\breve\Omega)$
\[
   \sum_{k\in\Z} \DNormc{\cF^k\ru}{L^2(\breve\Omega)} = \DNormc{\ru}{L^2(\breve\Omega)}.
\]
\end{proposition}

\begin{notation}
\label{no:HmbO}
We use the standard Sobolev spaces $H^m(\breve\Omega)$, $m\in\N$, with norm defined by 
\[
   \DNorm{\ru}{H^m(\breve\Omega)} = 
   \big(\sum_{|\alphab|\le m} \DNormc{\partial^{\alphab}_{\bx}\ru}{L^2(\breve\Omega)} \big)^{\frac12} \,
\]
and the usual multi-index notation $\partial^{\alphab}_{\bx} = \partial^{\alpha_1}_{x}\partial^{\alpha_2}_{y}\partial^{\alpha_3}_{z}$ , $|\alphab|=\alpha_{1}+\alpha_{2}+\alpha_{3}$ for $\alphab\in\N^3$.\\
In $H^m(\breve\Omega)$ introduce the seminorms
\begin{equation}
\label{eq:semiperp}
   \Normc{\ru}{H^m(\breve\Omega)} = \sum_{|\alphab|=m} \DNormc{\partial^{\alphab}_{\bx}\ru}{L^2(\breve\Omega)}
   \quad\mbox{and}\quad
   \Normc{\ru}{H^m_{\perp}(\breve\Omega)} =
   \sum_{\sz|\alphab|=m \atop \sz\alpha_3=0} \DNormc{\partial^{\alphab}_{\bx}\ru}{L^2(\breve\Omega)}
\end{equation}
The norm of $H^m(\breve\Omega)$ then satisfies by definition
\begin{equation}
\label{eq:DerivEnz}
   \DNormc{\ru}{H^m(\breve\Omega)} = \sum_{j=0}^m \Normc{\ru}{H^j(\breve\Omega)}
   = \sum_{j=0}^m \sum_{k=0}^j \Normc{\partial^{j-k}_z\ru}{H^k_{\perp}(\breve\Omega)}
\end{equation}
\end{notation}

\begin{remark}
\label{r:mazya1.1.18}
Since the Sobolev spaces $H^m(\breve\Omega)$ are the only function spaces studied in this paper, we are allowed to be vague about the inclusion (or not) of the axis of rotation $\cA=\{r=0\}$ into $\breve\Omega$. The reason is that $\cA$ has zero Hausdorff measure of dimension $2$ (or of dimension $1$ in the 2D situation). For this case, it follows from \cite[Theorem 1.1.18]{MazyaSob2011} that
\[
 H^{m}(\breve\Omega) = H^{m}(\breve\Omega\setminus\cA)
\]
in the sense that the distributional derivatives of order $\le m$ are the same, whether taken in the sense of distributions on $\breve\Omega$ or on $\breve\Omega\setminus\cA$ and, of course, $L^{2}(\breve\Omega)$ and $L^{2}(\breve\Omega\setminus\cA)$ are naturally identical. 
\end{remark}

\begin{notation}
\label{no:HmkbO}
Denote the space $\cF^k(L^2(\breve\Omega))$ by $H^0_{(k)}(\breve\Omega)$ and, more generally for any $m\in\N$, define the image of the Sobolev space $H^m(\breve\Omega)$
\[
   \cF^k(H^m(\breve\Omega)) =: H^m_{(k)}(\breve\Omega),
\]
with norm induced by the norm of $H^m(\breve\Omega)$.
\end{notation}

The main subject of this paper is the characterization of these spaces $H^m_{(k)}(\breve\Omega)$. 

\subsection{Partial derivatives}
The partial derivatives $\partial_x$ and $\partial_y$ mix Fourier components of different order, but there exist certain linear combinations that avoid this problem.

Consider the differential operators of order $1$
\begin{equation}
\label{eq:dzeta}
   \partial_\zeta = \tfrac{1}{\sqrt2}(\partial_x - i\partial_y),\quad
   \partial_{\bar\zeta} = \tfrac{1}{\sqrt2}(\partial_x + i\partial_y),
   \quad \mbox{and}\quad \partial_z.
\end{equation}
Here $\zeta$ stands for $\tfrac{1}{\sqrt2}(x+iy)$. There holds $\partial_\zeta \zeta = 1$, $\partial_{\bar\zeta} \bar\zeta = 1$, $\partial_\zeta \bar\zeta = 0$, and $\partial_{\bar\zeta} \zeta = 0$.

We have the following commutation formulas for any chosen $\theta\in\T$
\begin{equation}
\label{eq:comGf}
   \cG_\theta\circ \partial_{\zeta} = e^{-i\theta}\,\partial_{\zeta} \circ \cG_\theta \,,\quad
   \cG_\theta\circ \partial_{\bar\zeta} = e^{i\theta}\, \partial_{\bar\zeta} \circ \cG_\theta\,,
   \quad\mbox{and}\quad
   \cG_\theta\circ \partial_{z}  = \partial_{z} \circ \cG_\theta .
\end{equation}
We deduce
\[
\begin{aligned}
   \cF^k(\partial_{\zeta}\ru)
   &=
   \frac{1}{2\pi} \int_\T (\rG_\theta \partial_{\zeta}\ru)(\bx)\,e^{-ik\theta}\,\rd\theta \\
   &=
   \frac{1}{2\pi} \int_\T e^{-i\theta} (\partial_{\zeta} \rG_\theta \ru)(\bx)\,e^{-ik\theta}\,\rd\theta \\
   &=
   \frac{\partial_{\zeta}}{2\pi}  \int_\T ( \rG_\theta \ru)(\bx)\,e^{-i(k+1)\theta}\,\rd\theta
\end{aligned}
\]
and similarly for the other two. Hence the commutation formulas for the Fourier operators $\cF^k$:
\begin{equation}
\label{eq:comFk}
   \cF^k\circ \partial_{\zeta} = \partial_{\zeta} \circ\cF^{k+1},\quad
   \cF^k\circ \partial_{\bar\zeta} = \partial_{\bar\zeta} \circ\cF^{k-1},\quad\mbox{and}\quad
   \cF^k\circ \partial_{z} = \partial_{z} \circ\cF^{k}.
\end{equation}

This allows for simple formulas for the seminorms and norms in the Fourier spaces $H^m_{(k)}(\breve\Omega)$.

\begin{notation}
Denote by $\bz$ the triple of variables $(\zeta,\bar\zeta,z)$ and for any multi-index $\alphab\in\N^3$
\[
   \partial^{\alphab}_{\bz} = \partial^{\alpha_1}_{\zeta}\partial^{\alpha_2}_{\bar\zeta}\partial^{\alpha_3}_{z}.
\]
Likewise let $\zetab=(\zeta,\bar\zeta)$ and for $\betab\in\N^2$, set $\partial^\betab_\zetab = \partial^{\beta_1}_{\zeta}\partial^{\beta_2}_{\bar\zeta}$.

\end{notation}

Using the polarization identity in Hilbert space
\begin{equation}
\label{eq:pol}
   \DNormc{\frac{a+b}{\sqrt2}}{} + \ \DNormc{\frac{a-b}{\sqrt2}}{} = \
   \DNormc{a}{} + \ \DNormc{b}{}
\end{equation}
we can now rewrite the Sobolev norms and seminorms \eqref{eq:semiperp} in terms of the variables $\bz$.

\begin{lemma}
Let $m\in\N$. Then 
\begin{equation}
\label{eq:snormzzb}
   \Normc{\ru}{H^m(\breve\Omega)} 
   = \sum_{|\alphab|=m} \DNormc{\partial^\alphab_\bz\ru}{L^2(\breve\Omega)}
   \quad\mbox{and}\quad
   \Normc{\ru}{H^j_\perp(\breve\Omega)} =
   \sum_{|\betab|=j} \DNormc{\partial^{\betab}_{\zetab}\ru}{L^2(\breve\Omega)}
\end{equation}
\end{lemma}

\begin{proposition}
\label{pr:2}
Let $m\in\N$.

\noindent
(i) Let $k\in\N$. Then for any $\ru\in H^m_{(k)}(\breve\Omega)$
\[
   \Normc{\ru}{H^m(\breve\Omega)} = \Normc{\cF^k\ru}{H^m(\breve\Omega)} =
   \sum_{|\alphab|=m} \DNormc{\cF^{k-\alpha_1+\alpha_2}(\partial^{\alphab}_{\bz}\ru)}{L^2(\breve\Omega)}
\]
and also,
\[
   \Normc{\ru}{H^m(\breve\Omega)} = \sum_{j=0}^m \Normc{\cF^k(\partial^{m-j}_z\ru)}{H^j_\perp(\breve\Omega)}
   \quad\mbox{with}\quad
   \Normc{\cF^k\rv}{H^j_\perp(\breve\Omega)} =
   \sum_{|\betab|=j} \DNormc{\cF^{k-\beta_1+\beta_2}(\partial^{\betab}_{\zetab}\rv)}{L^2(\breve\Omega)}\,.
\]
(ii) The operators $\cF^k$ define orthogonal projections for the norm of $H^m(\breve\Omega)$ and the family of operators $\big(\cF^k\big)_{k\in\Z}$ defines an isometric isomorphism from $H^m(\breve\Omega)$ to the direct sum 
$\bigoplus_{k\in\Z}H^m_{(k)}(\breve\Omega)$: For any $\ru\in H^m(\breve\Omega)$
\begin{equation}
\label{eq:dirsum}
 \sum_{k\in\Z} \DNormc{\cF^k\ru}{H^m(\breve\Omega)} = \DNormc{\ru}{H^m(\breve\Omega)}.
\end{equation}
\end{proposition}

\begin{proof}
(i) Let $\ru\in H^m_{(k)}(\breve\Omega)$. By definition of this space, $\ru$ coincides with $\cF^k\ru$. Using formulas \eqref{eq:comFk} we obtain
\[
\begin{aligned}
   \Normc{\ru}{H^m(\breve\Omega)} &= \Normc{\cF^k\ru}{H^m(\breve\Omega)} \\
   &= \sum_{|\alphab|=m} \DNormc{\partial^\alphab_\bz (\cF^k\ru)}{L^2(\breve\Omega)} \\
   &= \sum_{|\alphab|=m} \DNormc{\cF^{k-\alpha_1+\alpha_2}(\partial^{\alphab}_{\bz}\ru)}{L^2(\breve\Omega)}
\end{aligned}
\]
(ii) Owing to the previous identity, the orthogonality in $H^m(\breve\Omega)$ is a consequence of the orthogonality in $L^2(\breve\Omega)$.
\end{proof}


\section{Fourier coefficients in the meridian domain}
\label{s:FC}

\subsection{Fourier coefficients}
Let us choose a meridian domain for $\breve\Omega$. We may take for instance the intersection of $\breve\Omega$ with the half-plane $x>0,\;y=0$:
\[
   \Omega = \breve\Omega\cap\{\bx\in\R^3\,:\quad \quad x>0,\;y=0\}
\]
and define cylindrical coordinates $\bx\mapsto\cT\bx=(r,z,\theta)\in\R_+\times\R\times\T$ with the following choice:
\begin{itemize}
\item If $\bx=(x,0,z)$ with $x>0$, then $r=x$ and $\theta=0$
\item Else there exists $\bx_0=(r,0,z)$ and $\theta\in\T$ such that $\bx=\cR_\theta(\bx_0)$,\\
i.e.\ $x=r\cos\theta$, $y=r\sin\theta$.
\end{itemize}
In other words $\breve\Omega$ is completely determined by its meridian domain $\Omega\subset\R_+\times\R$ 
and can be identified with the product  
$ \Omega\times\T$ by the change of variables
   $\bx\mapsto\big((r,z),\theta\big)$.

\begin{notation}
We define the natural weighted norm
\[
 \DNormc{u}{L^2_1(\Omega)} = 2\pi\int_\Omega |u(r,z)|^2 \,r\, dr \,dz
\]
and denote by $L^2_1(\Omega)$ the space
\[
   L^2_1(\Omega) = \{ u\in L^2_{\loc}(\Omega)\,:\quad \DNorm{u}{L^2_1(\Omega)} <\infty \}
\]
\end{notation}

Coming back to the Fourier coefficients, we can re-write formula \eqref{eq:inva} in the form
\[
   e^{-ik\theta}\,\rG_\theta (\cF^k\ru)(\bx) = \cF^k\ru(\bx),\quad  \theta\in\T,\;\bx\in\breve\Omega\,.
\]
This means that for any chosen $\bx_0$, the function
\[
   \theta\mapsto e^{-ik\theta}\,(\cF^k\ru)(\cR_\theta\bx_0)
\]
is constant. Choosing $\bx_0=(r,0,z)$, we obtain that the function
\[
   (r,z,\theta)\mapsto e^{-ik\theta}\,(\cF^k\ru)(\cR_\theta(r,0,z))
\]
is a function of $(r,z)$. Taking $\theta=0$, we see that this function coincides with the classical Fourier coefficient $u^k:(r,z)\mapsto u^k(r,z)$ of $\ru$ defined as:
\begin{equation}
\label{eq:FouCo}
   u^k(r,z) = \frac{1}{2\pi} \int_\T  \ru\left(\cT^{-1}(r,z,\theta)\right) \, e^{-ik\theta}\;
   \rd\theta,\quad k\in\Z.
\end{equation}

The following results are now straightforward.

\begin{lemma}
\label{lem:m0}
Let $\ru\in L^2(\breve\Omega)$. Its Fourier coefficients \eqref{eq:FouCo} satisfy the relations
\[
   e^{ik\theta} u^k(r,z) = (\cF^k\ru)(\bx)\quad\mbox{with}\quad \bx = \cT^{-1}(r,z,\theta),
\]
Each coefficient $u^k$ belongs to $L^2_1(\Omega)$ and
\[
   \DNormc{\cF^k\ru}{L^2(\breve\Omega)} = \DNormc{u^k}{L^2_1(\Omega)}.
\]
\end{lemma}

Now the question is to characterize the Fourier coefficients $u^k$ of a function $u\in H^m(\breve\Omega)$. Let us introduce relevant spaces for this.

\begin{notation}
\label{n:dirsum}
Let $m\in\N$ and $k\in\Z$. Writing $\rw_k(\bx) := e^{ik\theta} w(r,z)$ for $w\in L^2_1(\Omega)$, we define the spaces
\[
   H^m_{(k)}(\Omega) = \{w\in L^2_1(\Omega)\,:\quad \rw_k\in H^m(\breve\Omega)
   \}
\]
and, likewise, for spaces involving only derivatives in $x$ and $y$
\[
   H^m_{\perp(k)}(\Omega) = \{w\in L^2_1(\Omega)\,:\quad \rw_k\in H^m_\perp(\breve\Omega)
   \}\,.
\]
The norms and seminorms in $H^m_{(k)}(\Omega)$ and $H^m_{\perp(k)}(\Omega)$ are defined accordingly
\begin{equation}
\label{eq:normHmk}
   \DNorm{w}{H^m_{(k)}(\Omega)} := \DNorm{\rw_k}{H^m(\breve\Omega)},\quad
   \Norm{w}{H^m_{(k)}(\Omega)} := \Norm{\rw_k}{H^m(\breve\Omega)},\quad
   \Norm{w}{H^m_{\perp(k)}(\Omega)} := \Norm{\rw_k}{H^m_\perp(\breve\Omega)}.
\end{equation}
\end{notation}

\begin{remark}
\label{r:dirsum}
We can write more formally this correspondence $w\mapsto \breve w_k$ by introducing the operator $\cT_*$ of change of variables to cylindrical coordinates: for a function $\breve u$ defined on $\breve\Omega$, $\cT_*\ru$ is the function defined on $\Omega\times\T$ by
\[
   (\cT_*\ru)(r,z,\theta) = \ru(\bx)\quad\mbox{with}\quad \bx=\cT^{-1}(r,z,\theta).
\]
By definition, the operator
\[
   w\longmapsto \breve w_k = \cT_*^{-1}(e^{ik\theta} w)
\]
is then an isometry from $H^m_{(k)}(\Omega)$ to $H^m(\breve\Omega)$.
Hence, as a consequence of \eqref{eq:dirsum}, we have
\begin{equation}
\label{eq:dirsum2}
   \DNormc{\ru}{H^m(\breve\Omega)} =
   \sum_{k\in\Z} \DNormc{u^k}{H^m_{(k)}(\Omega)}.
\end{equation}
\end{remark}

\begin{remark}
\label{rem:semi0}
{\em (i)} We have
\begin{equation}
\label{eq:semi1}
   \DNormc{w}{H^m_{(k)}(\Omega)} = \sum_{j=0}^m \Normc{w}{H^j_{(k)}(\Omega)}
\end{equation}
and
\begin{equation}
\label{eq:semi2}
   \Normc{w}{H^m_{(k)}(\Omega)} = \sum_{j=0}^m \Normc{\partial^{m-j}_z w}{H^j_{\perp(k)}(\Omega)}
\end{equation}
Hence it suffices to characterize the seminorms $H^m_{\perp(k)}(\Omega)$ for any $m\in\N$.

\noindent {\em (ii)} For 2D axisymmetric domains, the meridian domain is an interval in $\R_+$, hence the seminorms $H^j_{(k)}(\Omega)$ and $H^j_{\perp(k)}(\Omega)$ are the same.
\end{remark}

\subsection{Partial derivatives in cylindrical coordinates}
\mbox{ } 
We write the operators $\partial_\zeta$ and $\partial_{\bar\zeta}$ \eqref{eq:dzeta} in polar coordinates 
\begin{equation}
\label{eq:1}
   \partial_\zeta = \tfrac{1}{\sqrt2}(\partial_x - i\partial_y) = 
   \frac{e^{-i\theta}}{\sqrt2}\Big(\partial_r-\frac{i\partial_\theta}{r}\Big)\quad\mbox{and}\quad
   \partial_{\bar\zeta} = \tfrac{1}{\sqrt2}(\partial_x + i\partial_y) =
   \frac{e^{i\theta}}{\sqrt2}\Big(\partial_r+\frac{i\partial_\theta}{r}\Big)
\end{equation}

Now, choose $w\in L^2_1(\Omega)$, fix $k\in\Z$ and write $\rw$ instead of $\rw_k$ for the function $\bx\mapsto e^{ik\theta} w(r,z)$.
Then $\cF^k\rw=\rw$.  We set
\begin{equation}
\label{eq:u+-}
   \ru_- = \partial_\zeta\rw \quad\mbox{and}\quad \ru_+ = \partial_{\bar\zeta}\rw.
\end{equation}
Then we deduce from \eqref{eq:1} the equalities
\begin{equation}
\label{eq:Fu+-}
   \cF^{k-1}\ru_-=\ru_- \quad\mbox{and}\quad \cF^{k+1}\ru_+=\ru_+,
\end{equation} 
together with the formulas for the corresponding Fourier coefficients $[\ru_\pm]^{k\pm1}$ of $\ru_\pm$
\begin{equation}
\label{eq:2}
   [\ru_-]^{k-1} = \frac{1}{\sqrt2}\left(\partial_r +\frac{k}{r}\right)\,w
   \quad\mbox{and}\quad
   [\ru_+]^{k+1} = \frac{1}{\sqrt2}\left(\partial_r -\frac{k}{r}\right)\,w.
\end{equation}

\subsection{Lowest order Sobolev norms}
Before formulating the general result, we first consider the Sobolev norms of order $m\le2$. We continue to use the notation introduced in the preceding subsection, in particular for a fixed $k\in\Z$, 
$
 \rw(\bx) = e^{ik\theta} w(r,z).
$

\subsubsection{Case $m=0$} Here we have
\[
 \Normc{w}{H^0_{\perp(k)}(\Omega)} = \DNormc{\rw}{L^2(\breve\Omega)}
  = \DNormc{w}{L^2_1(\Omega)} \,.
\]

\subsubsection{Case $m=1$}
Here we have 
\[
\begin{aligned}
   \Normc{w}{H^1_{\perp(k)}(\Omega)} 
\overset{\eqref{eq:normHmk}}{=} \Normc{\rw}{H^1_\perp(\breve\Omega)}
&\overset{\eqref{eq:snormzzb}}{=} \DNormc{\partial_\zeta\rw}{L^2(\breve\Omega)}
   +\DNormc{\partial_{\bar\zeta}\rw}{L^2(\breve\Omega)}
   \\
&\overset{\eqref{eq:u+-}}{=} \DNormc{\ru_-}{L^2(\breve\Omega)}
   +\DNormc{\ru_+}{L^2(\breve\Omega)}
   \\
&\overset{\eqref{eq:Fu+-}}{=} \DNormc{[\ru_-]^{k-1}}{L^2_1(\Omega)}
   +\DNormc{[\ru_+]^{k+1}}{L^2_1(\Omega)}
   \\
&\overset{\eqref{eq:2}}{=}  \tfrac12 \DNormc{\left(\partial_r  + \tfrac{k}{r}\right)w}{L^2_1(\Omega)}
   +\tfrac12 \DNormc{\left(\partial_r  - \tfrac{k}{r}\right)w}{L^2_1(\Omega)}
   \\
&\overset{\eqref{eq:pol}}{=}  \DNormc{\partial_r w}{L^2_1(\Omega)}
   +\DNormc{\tfrac{k}{r}\, w}{L^2_1(\Omega)}
\end{aligned}
\]
We note the result, when $m=1$:
\begin{equation}
\label{eq:m1}
   \Normc{w}{H^1_{\perp(k)}(\Omega)}
   = \DNormc{\partial_r w}{L^2_1(\Omega)}
   +\DNormc{\tfrac{k}{r}\, w}{L^2_1(\Omega)} \,.
\end{equation}
We see that the seminorms for $k=0$ and for $k\ne0$ are not equivalent, and that those for $k\ne0$ are all equivalent, but not uniformly in $k$.

\subsubsection{Case $m=2$}
Here we write 
\[
\begin{aligned}
   \Normc{w}{H^2_{\perp(k)}(\Omega)} 
&\overset{\eqref{eq:normHmk}}{=} \Normc{\rw}{H^2_\perp(\breve\Omega)}
   \\
&\overset{\eqref{eq:snormzzb}}{=} \DNormc{\partial_\zeta\rw}{H^1_\perp(\breve\Omega)}
   +\DNormc{\partial_{\bar\zeta}\rw}{H^1_\perp(\breve\Omega)}
   \\
&\overset{\eqref{eq:u+-}}{=} \DNormc{\ru_-}{H^1_\perp(\breve\Omega)}
   +\DNormc{\ru_+}{H^1_\perp(\breve\Omega)}
   \\
&\overset{\eqref{eq:Fu+-}}{=}  \DNormc{\partial_r [\ru_-]^{k-1} }{L^2_1(\Omega)}
   +\DNormc{\tfrac{k-1}{r}\, [\ru_-]^{k-1} }{L^2_1(\Omega)} \\
   & \quad
   +\DNormc{\partial_r [\ru_+]^{k+1} }{L^2_1(\Omega)}
   +\DNormc{\tfrac{k+1}{r} [\ru_+]^{k+1} }{L^2_1(\Omega)}
   \\
&\overset{\eqref{eq:2}}{=}  
   \tfrac12 \DNormc{\partial_r \left(\partial_r +\tfrac{k}{r}\right) w }{L^2_1(\Omega)}
   +\tfrac12\DNormc{\tfrac{k-1}{r}\, \left(\partial_r +\tfrac{k}{r}\right) w }{L^2_1(\Omega)} \\
   & \quad +
   \tfrac12\DNormc{\partial_r \left(\partial_r -\tfrac{k}{r}\right) w }{L^2_1(\Omega)}
   +\tfrac12\DNormc{\tfrac{k+1}{r} \left(\partial_r -\tfrac{k}{r}\right) w }{L^2_1(\Omega)}    \\
&=: A_1 + A_2 + B_1 + B_2
\end{aligned}
\]
From the polarization identity \eqref{eq:pol} follows
\[
   A_1+B_1 = \DNormc{\partial^2_r w }{L^2_1(\Omega)}
   + \DNormc{\partial_r \left(\tfrac{k}{r}\right) w }{L^2_1(\Omega)}
\]
and
\[
   A_2+B_2 = \DNormc{\left(\tfrac{k}{r}\,\partial_r -\tfrac{k}{r^2}\right) w }{L^2_1(\Omega)}
   + \DNormc{\left(\tfrac{1}{r}\,\partial_r -\tfrac{k^2}{r^2}\right) w }{L^2_1(\Omega)}
\]
Note that
\[
   \partial_r \left(\tfrac{k}{r}\right) w = \left(\tfrac{k}{r}\,\partial_r -\tfrac{k}{r^2}\right) w.
\]
Therefore we have found that
\[
   \Normc{w}{H^2_{\perp(k)}(\Omega)} \cong
   \DNormc{\partial^2_r w }{L^2_1(\Omega)}
   + \DNormc{\partial_r \left(\tfrac{k}{r}\right) w }{L^2_1(\Omega)}
   + \DNormc{\left(\tfrac{1}{r}\,\partial_r -\tfrac{k^2}{r^2}\right) w }{L^2_1(\Omega)}.
\]
Considering separately the cases $k=0$, $|k|=1$, and $|k|\ge2$, we deduce
\begin{equation}
\label{eq:m2}
   \Normc{w}{H^2_{\perp(k)}(\Omega)} \cong
   \begin{cases}
   \DNormc{\partial^2_r w }{L^2_1(\Omega)} + \DNormc{\frac{1}{r}\,\partial_r w }{L^2_1(\Omega)}
   &\mbox{if}\ k=0 \\[0.8ex]
   \DNormc{\partial^2_r w }{L^2_1(\Omega)} + \DNormc{\partial_r \left(\frac{1}{r}\right) w }{L^2_1(\Omega)}
   &\mbox{if}\ |k|=1 \\[0.8ex]
   \DNormc{\partial^2_r w }{L^2_1(\Omega)}
   + \DNormc{\partial_r \left(\frac{k}{r}\right) w }{L^2_1(\Omega)}
   + \DNormc{\left(\frac{k}{r}\right)^2 w }{L^2_1(\Omega)}
   &\mbox{if}\ |k|\ge2 \\
   \end{cases}
\end{equation}
where we use the notation $\cong$ to denote equivalence uniform both with respect to the order $k$ of the Fourier expansion and with respect to the domain $\Omega$. For further reference we make this explicit: 
\begin{notation}
For two families of seminorms $\Norm{\cdot}{A^m_{(k)}(\Omega)}$ and $\Norm{\cdot}{B^m_{(k)}(\Omega)}$ we write
\[
   \Norm{\cdot}{A^m_{(k)}(\Omega)} \cong \Norm{\cdot}{B^m_{(k)}(\Omega)}
\]
if for each $m\in\N$, there exists a constant $C=C(m)\ge1$ independent of $\Omega$ such that for all $k\in\Z$ and all $w$
\[
   C^{-1} \Norm{w}{A^m_{(k)}(\Omega)} \le \Norm{w}{B^m_{(k)}(\Omega)} \le C \Norm{w}{A^m_{(k)}(\Omega)}.
\]
\end{notation}


\section{Main result}
\label{s:Main}

\subsection{Statement}
The previous results for $m=0$, $m=1$, and $m=2$ suggest the introduction of the following two seminorms for general $m$.

\begin{notation}
For $m\in\N$ and $k\in\Z$, let $\Norm{\cdot}{W^m_{(k)}(\Omega)}$ and $\Norm{\cdot}{X^m_{(k)}(\Omega)}$ be the seminorms defined as
\begin{equation}
\label{eq:Wmk}
   \Normc{w}{W^m_{(k)}(\Omega)} =
   \sum_{\ell=0}^{\min\{|k|,m\}} \DNormc{\partial_{r}^{m-\ell}\left(\frac{|k|}{r}\right)^{\ell} w }{L^2_1(\Omega)}
\end{equation}
where we understand that if $\ell=0$:
\begin{equation}
\label{eq:ell0}
   \left(\frac{|k|}{r}\right)^{0}\equiv1,\quad \forall k\in\Z,
\end{equation}
and
\begin{equation}
\label{eq:Xmk}
   \Normc{w}{X^m_{(k)}(\Omega)} =
   \sum_{\ell=1}^{[(m-|k|)/2]}
      \DNormc{\partial_{r}^{m-|k|-2\ell}\left(\frac{1}{r}\partial_{r}\right)^{\ell}
      \left(\frac{1}{r}\right)^{|k|} w }{L^2_1(\Omega)}
\end{equation}
\end{notation}

\begin{remark}
The seminorm $\Norm{\cdot}{W^m_{(k)}(\Omega)}$ is a weighted seminorm.
\begin{itemize}
\item[\emph{(i)}] If $|k|\ge m$, it is (uniformly) equivalent to the standard weighted seminorm of Kondratev's type:
\[
   \Normc{w}{W^m_{(k)}(\Omega)} \cong
   \sum_{\ell=0}^{m} \DNormc{\left(\frac{|k|}{r}\right)^{\ell} \partial_{r}^{m-\ell} w }{L^2_1(\Omega)}
\]
for which the weight is on the left side of derivatives.
This is easily seen by commuting the derivatives $\partial_{r}$ with the multiplications by $\frac{|k|}{r}$, compare also \eqref{eq:com1} below.
\item[\emph{(ii)}] If $|k|\le m$, keeping in mind convention \eqref{eq:ell0}, we find that
\begin{equation}
\label{eq:Wk}
   \Normc{w}{W^m_{(k)}(\Omega)} \cong
   \sum_{\ell=0}^{|k|} \DNormc{\partial_{r}^{m-\ell}\left(\frac{1}{r}\right)^{\ell} w }{L^2_1(\Omega)}
\end{equation}
\end{itemize}
\end{remark}

\begin{remark}
The seminorm ${X^m_{(k)}(\Omega)}$ is a compound weighted seminorm with a non trivial structure, that is nonzero only if $|k|\le m-2$. Unlike the seminorm ${W^m_{(k)}(\Omega)}$, the weight on $w$ does not depend on $\ell$. Notice that the term of $\Norm{\cdot}{X^m_{(k)}(\Omega)}$ that would correspond to $\ell=0$ is in fact the term of $\Norm{\cdot}{W^m_{(k)}(\Omega)}$ corresponding to $\ell=|k|$.
\end{remark}

Now our main result can be formulated in terms of elementary seminorms as

\begin{theorem}
\label{th:semi}
Let $m\in\N$ and $k\in\Z$.
Let $\Omega\subset\R_+\times\R$ be the meridian domain of a 3D axisymmetric domain $\breve\Omega$.
We have the uniform equivalence of seminorms
\begin{equation}
\label{eq:equiv}
   \Normc{w}{H^m_{\perp (k)}(\Omega)} \cong \Normc{w}{W^m_{(k)}(\Omega)} + \Normc{w}{X^m_{(k)}(\Omega)}
\end{equation}
where equivalence constants do not depend on $k$ nor on $\Omega$.

The same result holds if $\Omega\subset\R_+$ is the meridian domain of a 2D axisymmetric domain $\breve\Omega$, with the convention that the seminorms $H^m_{\perp(k)}(\Omega)$ and $H^m_{(k)}(\Omega)$ are the same.
\end{theorem}

As noted in Remark~\ref{rem:semi0}, once the essential result \eqref{eq:equiv} concerning the seminorms is proved, obtaining the complete result of Theorem~\ref{t:main} concerning the equivalence of norms is a simple matter of adding up the seminorms and reinserting the axial derivatives $\partial_{z}$.

\begin{corollary}
\label{cor:thsemi}
Let $m\in\N$ and $k\in\Z$.

{\em (i)} Let $\Omega\subset\R_+\times\R$ be the meridian domain of a 3D axisymmetric domain $\breve\Omega$.
We have the uniform equivalence of norms
\begin{equation}
\label{eq:equiv2}
   \DNormc{w}{H^m_{(k)}(\Omega)} \cong \sum_{0\le p+q\le m}
   \left(\Normc{\partial^p_z w}{W^q_{(k)}(\Omega)} + \Normc{\partial^p_z w}{X^q_{(k)}(\Omega)}\right)
\end{equation}
where equivalence constants depend neither on $k$ nor on $\Omega$.

{\em (ii)} Let $\Omega\subset\R_+$ be the meridian domain of a 2D axisymmetric domain $\breve\Omega$.
We have the uniform equivalence of norms
\begin{equation}
\label{eq:equiv1}
   \DNormc{w}{H^m_{(k)}(\Omega)} \cong \sum_{0\le q\le m}
   \left(\Normc{w}{W^q_{(k)}(\Omega)} + \Normc{w}{X^q_{(k)}(\Omega)}\right)
\end{equation}
where equivalence constants depend neither on $k$ nor on $\Omega$.
\end{corollary}

The right hand sides of \eqref{eq:equiv2} and \eqref{eq:equiv1} amount to the same as the squared norms $\DNormc{w}{C^m_{(k)}(\Omega)}$ defined in \eqref{e:Cmk},
and hence Theorem \ref{t:main} is a consequence of Theorem \ref{th:semi}.

The rest of this section is devoted to the proof of Theorem \ref{th:semi}.

\subsection{Starting the proof}
The proof uses induction on the integer $m$.

\subsubsection{Initialization} We first notice that
\[
   \Norm{w}{X^0_{(k)}(\Omega)} = \Norm{w}{X^1_{(k)}(\Omega)} = 0,\quad\forall k\in\Z
\]
and that
\[
   \Norm{w}{W^0_{(k)}(\Omega)} = \DNorm{w}{L^2_1(\Omega)},\quad\forall k\in\Z
\]
whereas
\[
   \Normc{w}{W^1_{(k)}(\Omega)} =
   \DNormc{\partial_{r} w }{L^2_1(\Omega)} + \DNormc{\left(\frac{|k|}{r}\right) w }{L^2_1(\Omega)}
   ,\quad\forall k\in\Z
\]
Then Lemma \ref{lem:m0} yields \eqref{eq:equiv} for $m=0$, and \eqref{eq:m1} yields \eqref{eq:equiv} for $m=1$.

The value $m=2$ is the first one for which the seminorm $X^m_{(k)}(\Omega)$ is not identically $0$. We have 
\[
   \Norm{w}{X^2_{(0)}(\Omega)} = \DNormc{\frac{1}{r}\partial_{r} w }{L^2_1(\Omega)}
   \quad\mbox{and}\quad \Norm{w}{X^2_{(k)}(\Omega)} = 0\ \ \mbox{if} \ \ k\neq0.
\]
Then we see that \eqref{eq:m2} yields \eqref{eq:equiv} for $m=2$.

\subsubsection{Induction}
We now assume that $m\ge2$ and that \eqref{eq:equiv} holds for this value of $m$ and any $k\in\Z$. We wish to prove that \eqref{eq:equiv} holds for $m+1$ and any $k\in\Z$.

Let us choose $k\in\Z$. For symmetry reasons, we can suppose that $k\ge0$. Choose $w\in H^{m+1}_{\perp (k)}(\Omega)$. Set $\rw(\bx)=e^{ik\theta}w(r,z)$ and according to \eqref{eq:u+-} define
\[
   \ru_- = \partial_\zeta\rw \quad\mbox{and}\quad \ru_+ = \partial_{\bar\zeta}\rw.
\]
Then $\cF^{k-1}\ru_-=\ru_-$ and $\cF^{k+1}\ru_+=\ru_+$. Moreover \eqref{eq:2} holds, so we can write
\begin{equation}
\begin{cases}
\label{eq:3}
   \ru_- = e^{i(k-1)\theta} w_-(r,z)\quad&\mbox{with}\quad
   w_-(r,z) = \frac{1}{\sqrt2}\left(\partial_r +\frac{k}{r}\right) w \in H^{m}_{\perp (k-1)}(\Omega)
   \\[1ex]
   \ru_+ = e^{i(k+1)\theta} w_+(r,z)\quad&\mbox{with}\quad
   w_+(r,z) = \frac{1}{\sqrt2}\left(\partial_r -\frac{k}{r}\right) w \in H^{m}_{\perp (k+1)}(\Omega)
\end{cases}
\end{equation}
(so $w_{\pm}$ are abbreviations for $[\ru_\pm]^{k\pm1}$). We have
\begin{align*}
   \Normc{w}{H^{m+1}_{\perp(k)}(\Omega)} &= \Normc{\rw}{H^{m+1}_\perp(\breve\Omega)}
   \nonumber\\
    &= \Normc{\partial_\zeta\rw}{H^m_\perp(\breve\Omega)}
   +\Normc{\partial_{\bar\zeta}\rw}{H^m_\perp(\breve\Omega)}
   \\
&= \Normc{w_-}{H^m_{\perp(k-1)}(\Omega)}
   +\Normc{w_+}{H^m_{\perp(k+1)}(\Omega)}
\end{align*}
By the recurrence hypothesis
\[
   \Normc{w_-}{H^m_{\perp(k-1)}(\Omega)} \cong \Normc{w_-}{W^m_{(k-1)}(\Omega)} + \Normc{w_-}{X^m_{(k-1)}(\Omega)}
\]
and
\[
   \Normc{w_+}{H^m_{\perp(k+1)}(\Omega)} \cong \Normc{w_+}{W^m_{(k+1)}(\Omega)} + \Normc{w_+}{X^m_{(k+1)}(\Omega)}
\]
Thus, we find that
\[
   \Normc{w}{H^{m+1}_{\perp(k)}(\Omega)} \cong
   \Normc{w_-}{W^m_{(k-1)}(\Omega)} + \Normc{w_-}{X^m_{(k-1)}(\Omega)}
   +\Normc{w_+}{W^m_{(k+1)}(\Omega)} + \Normc{w_+}{X^m_{(k+1)}(\Omega)} \,.
\]
Hence the proof of Theorem \ref{th:semi} reduces to proving for any $m\ge2$
\begin{equation}
\label{eq:4}
   \sum_{\varepsilon\in\{\pm1\}} \Normc{w_\varepsilon}{W^m_{(k+\varepsilon)}(\Omega)} +
   \sum_{\varepsilon\in\{\pm1\}} \Normc{w_\varepsilon}{X^m_{(k+\varepsilon)}(\Omega)}
   \cong
   \Normc{w}{W^{m+1}_{(k)}(\Omega)} +
   \Normc{w}{X^{m+1}_{(k)}(\Omega)}
\end{equation}
with $w_\varepsilon=w_\pm$ for $\varepsilon=\pm1$ as defined in \eqref{eq:3}.

We split the proof of \eqref{eq:4} in four parts I, II, III, and IV, according to the cases $k\ge m+1$, $k=m$, $k\in\{1,\ldots,m-1\}$, and $k=0$.

\subsection{Proof of induction step}
\subsubsection*{\textbf{Part I}: Case $k\ge m+1$}
Then
\begin{equation}
\label{eq:P1}
   \Normc{w_-}{X^m_{(k-1)}(\Omega)} = \Normc{w_+}{X^m_{(k+1)}(\Omega)} = 0
\end{equation}
and
\[
   \Normc{w_\pm}{W^m_{(k\pm1)}(\Omega)} =
   \sum_{\ell=0}^m \DNormc{\partial_{r}^{m-\ell}\left(\tfrac{k\pm1}{r}\right)^{\ell} w_\pm }{L^2_1(\Omega)}
\]
Since $k\ge2$, we have
\[
   \tfrac12\, k \le k\pm1 \le 2k
\]
Hence
\[
   \Normc{w_\pm}{W^m_{(k\pm1)}(\Omega)} \cong
   \sum_{\ell=0}^m \DNormc{\partial_{r}^{m-\ell}\left(\tfrac{k}{r}\right)^{\ell} w_\pm }{L^2_1(\Omega)}
\]
Using the expressions \eqref{eq:3} of $w_\pm$ and the polarization identity \eqref{eq:pol} we find
\begin{equation}
\label{eq:P2}
   \sum_{\varepsilon\in\{\pm1\}} \Normc{w_\varepsilon}{W^m_{(k+\varepsilon)}(\Omega)} \cong
   \sum_{\ell=0}^m \left(\DNormc{\partial_{r}^{m-\ell}\left(\tfrac{k}{r}\right)^{\ell} \partial_r w }{L^2_1(\Omega)}
   \!\!+
   \DNormc{\partial_{r}^{m-\ell}\left(\tfrac{k}{r}\right)^{\ell} \left(\tfrac{k}{r}\right) w }{L^2_1(\Omega)}\right)
\end{equation}
We note the commutation relation
\begin{equation}
\label{eq:com1}
    \left(\frac{k}{r}\right)^{\ell}\partial_{r}w
    =\partial_{r}\left(\frac{k}{r}\right)^{\ell}w
    +\frac{\ell}{k}\left(\frac{k}{r}\right)^{\ell+1}w,\quad 1\leq \ell\leq m
\end{equation}

\begin{lemma}
\label{lem:ineq}
Let $a$ and $b$ belong to a normed space, and $\tau\in\R$. Then
\begin{equation}
\label{eq:ineq}
   \frac{1}{2+\tau^2} \,\left(\DNormc{a}{} + \DNormc{b}{} \right) \le
   \DNormc{a+\tau b}{} + \DNormc{b}{} \le
   (2+\tau^2)\left(\DNormc{a}{} + \DNormc{b}{} \right).
\end{equation}
\end{lemma}

\begin{proof}
It suffices to chain the following inequalities
\[
\begin{aligned}
   \DNormc{a+\tau b}{} &\le \DNormc{a}{} + 2|\tau|\,\DNorm{a}{}\,\DNorm{b}{} + \tau^2\DNormc{b}{}
   \\
   &\le \DNormc{a}{} + |\tau|\,\left(|\tau|\DNormc{a}{} + |\tau|^{-1}\DNormc{b}{}\right) + \tau^2\DNormc{b}{}
   \\
   &\le (1+\tau^2)\left(\DNormc{a}{} + \DNormc{b}{} \right)
\end{aligned}
\]
Whence the right part of inequality \eqref{eq:ineq}. We obtain the left part of \eqref{eq:ineq} by applying its right part to $c:=a+\tau b$ and $b$:
\[
   \DNormc{a}{} + \DNormc{b}{}  =  \DNormc{c-\tau b}{} + \DNormc{b}{}\le
   (2+\tau^2)\left(\DNormc{c}{} + \DNormc{b}{} \right)
\]
The lemma is proved.
\end{proof}

Since $|\frac{\ell}{k}|\le1$, combining \eqref{eq:P2} and \eqref{eq:com1} we deduce from Lemma \ref{lem:ineq}
\begin{align}
   \sum_{\varepsilon\in\{\pm1\}} \Normc{w_\varepsilon}{W^m_{(k+\varepsilon)}(\Omega)} &\cong
   \sum_{\ell=0}^m \left( \DNormc{\partial_{r}^{m+1-\ell}\left(\tfrac{k}{r}\right)^{\ell} w }{L^2_1(\Omega)}
   +
   \DNormc{\partial_{r}^{m-\ell}\left(\tfrac{k}{r}\right)^{\ell+1} w }{L^2_1(\Omega)} \right)
   \nonumber\\
   &\cong
   \sum_{\ell=0}^{m+1} \DNormc{\partial_{r}^{m+1-\ell}\left(\tfrac{k}{r}\right)^{\ell} w }{L^2_1(\Omega)}
\label{eq:P3}
\end{align}
Returning to \eqref{eq:P1} we have found that
\[
   \sum_{\varepsilon\in\{\pm1\}} \Normc{w_\varepsilon}{W^m_{(k+\varepsilon)}(\Omega)}+
   \sum_{\varepsilon\in\{\pm1\}} \Normc{w_\varepsilon}{X^m_{(k+\varepsilon)}(\Omega)} \cong
   \sum_{\ell=0}^{m+1} \DNormc{\partial_{r}^{m+1-\ell}\left(\tfrac{k}{r}\right)^{\ell} w }{L^2_1(\Omega)}
\]
which coincides with the seminorm $\Normc{w}{W^{m+1}_{(k)}(\Omega)}$. As the seminorm $X^{m+1}_{(k)}(\Omega)$ is zero for $k\ge m$, we have finally proved \eqref{eq:4} for $k\ge m+1$.

\subsubsection*{\textbf{Part II}: Case $k = m$}
This case is very similar to the previous one. The sole difference is that now $\Normc{w_-}{W^m_{(k-1)}(\Omega)}$ has one term less:
\[
   \Normc{w_-}{W^m_{(k-1)}(\Omega)} \cong
   \sum_{\ell=0}^{m-1} \DNormc{\partial_{r}^{m-\ell}\left(\tfrac{k}{r}\right)^{\ell} w_- }{L^2_1(\Omega)}
\]
Then instead of \eqref{eq:P2}--\eqref{eq:P3} we find
\begin{align}
   \sum_{\varepsilon\in\{\pm1\}} \Normc{w_\varepsilon}{W^m_{(k+\varepsilon)}(\Omega)}
   &\cong
   \sum_{\ell=0}^{m-1}
   \left(\DNormc{\partial_{r}^{m-\ell}\left(\tfrac{k}{r}\right)^{\ell} \partial_r w }{L^2_1(\Omega)}
   \!\!+
   \DNormc{\partial_{r}^{m-\ell}\left(\tfrac{k}{r}\right)^{\ell+1} w }{L^2_1(\Omega)}\right)
   + \DNormc{\left(\tfrac{k}{r}\right)^{m} w_+ }{L^2_1(\Omega)}
   \nonumber\\
   &\cong
   \sum_{\ell=0}^{m-1}\left( \DNormc{\partial_{r}^{m+1-\ell}\left(\tfrac{k}{r}\right)^{\ell} w }{L^2_1(\Omega)}
   \!\!+
   \DNormc{\partial_{r}^{m-\ell}\left(\tfrac{k}{r}\right)^{\ell+1} w }{L^2_1(\Omega)} \right)
   + \DNormc{\left(\tfrac{k}{r}\right)^{m} w_+ }{L^2_1(\Omega)}
   \nonumber\\
   & \cong
   \sum_{\ell=0}^{m} \DNormc{\partial_{r}^{m+1-\ell}\left(\tfrac{k}{r}\right)^{\ell} w }{L^2_1(\Omega)}
   + \DNormc{\left(\tfrac{k}{r}\right)^{m} w_+ }{L^2_1(\Omega)}
\label{eq:P4}
\end{align}
Now, we use the expression \eqref{eq:3} of $w_+$ (for $k=m$)
\[
   \DNormc{\left(\tfrac{k}{r}\right)^{m} w_+ }{L^2_1(\Omega)} \cong
   \DNormc{\left(\tfrac{k}{r}\right)^{m} (\partial_r - \tfrac{m}{r}) w }{L^2_1(\Omega)}
\]
The commutation relation \eqref{eq:com1} for $\ell=k=m$ yields
\[
   \left(\tfrac{k}{r}\right)^{m} (\partial_r - \tfrac{m}{r}) w  =
   \partial_r\left(\tfrac{k}{r}\right)^{m} w
\]
which means that the term $\DNormc{\left(\tfrac{k}{r}\right)^{m} w_+ }{L^2_1(\Omega)}$ is already contained in the sum over $\ell$ (for $\ell=m$) inside \eqref{eq:P4}.
Finally
\[
   \sum_{\varepsilon\in\{\pm1\}} \Normc{w_\varepsilon}{W^m_{(k+\varepsilon)}(\Omega)} \cong
   \sum_{\ell=0}^{m} \DNormc{\partial_{r}^{m+1-\ell}\left(\tfrac{k}{r}\right)^{\ell} w }{L^2_1(\Omega)}
\]
The sum on the right coincides with $\Normc{w}{W^{m+1}_{(m)}(\Omega)}$.
Therefore this proves \eqref{eq:4} for $k = m$ since $\Normc{w_\varepsilon}{X^m_{(m+\varepsilon)}(\Omega)}=0$ and $\Normc{w}{X^{m+1}_{(m)}(\Omega)}=0$.

\subsubsection*{\textbf{Part III}: Case $1\le k \le m-1$}
Now we can use the form \eqref{eq:Wk} of the seminorm $W^m_{(k)}(\Omega)$ and write
\[
   \Normc{w_-}{W^m_{(k-1)}(\Omega)} =
   \sum_{\ell=0}^{k-1} \DNormc{\partial_{r}^{m-\ell}\left(\tfrac{1}{r}\right)^{\ell} w_- }{L^2_1(\Omega)}
   \:\:\mbox{and}\;\;
   \Normc{w_+}{W^m_{(k+1)}(\Omega)} =
   \sum_{\ell=0}^{k+1} \DNormc{\partial_{r}^{m-\ell}\left(\tfrac{1}{r}\right)^{\ell} w_+ }{L^2_1(\Omega)}
\]
Then, along the same lines as in \eqref{eq:P2}--\eqref{eq:P3}, we find
\begin{align}
   \sum_{\varepsilon\in\{\pm1\}} \Normc{w_\varepsilon}{W^m_{(k+\varepsilon)}(\Omega)}
   &\cong
   \sum_{\ell=0}^{k-1}
   \left(\DNormc{\partial_{r}^{m-\ell}\left(\tfrac{1}{r}\right)^{\ell} \partial_r w }{L^2_1(\Omega)}
   \!\!+
   \DNormc{\partial_{r}^{m-\ell}\left(\tfrac{1}{r}\right)^{\ell+1} w }{L^2_1(\Omega)}\right)
   \nonumber\\
   &\qquad\qquad +  \sum_{\ell=k}^{k+1}
   \DNormc{\partial_{r}^{m-\ell}\left(\tfrac{1}{r}\right)^{\ell} w_+ }{L^2_1(\Omega)}
   \nonumber\\
   & \cong
   \sum_{\ell=0}^{k} \DNormc{\partial_{r}^{m+1-\ell}\left(\tfrac{1}{r}\right)^{\ell} w }{L^2_1(\Omega)}
   + \sum_{\ell=k}^{k+1}
   \DNormc{\partial_{r}^{m-\ell}\left(\tfrac{1}{r}\right)^{\ell} w_+ }{L^2_1(\Omega)}
\label{eq:P5}
\end{align}
Now
\begin{equation}
\label{eq:w+}
   \left(\tfrac{1}{r}\right)^{k} w_+ = \left(\tfrac{1}{r}\right)^{k} (\partial_r w - \tfrac{k}{r} w)
   = \partial_r \left(\tfrac{1}{r}\right)^{k} w
\end{equation}
Hence
\[
   \sum_{\ell=k}^{k+1}
   \DNormc{\partial_{r}^{m-\ell}\left(\tfrac{1}{r}\right)^{\ell} w_+ }{L^2_1(\Omega)}  =
   \DNormc{\partial_{r}^{m+1-k}\left(\tfrac{1}{r}\right)^{k} w }{L^2_1(\Omega)} +
   \DNormc{\partial_{r}^{m-k-1} \tfrac1r\partial_r\left(\tfrac{1}{r}\right)^{k} w }{L^2_1(\Omega)}
\]
We notice that the first term in the right is the same as the term $\ell=k$ in the sum \eqref{eq:P5}.
Therefore we have found
\begin{equation}
\label{eq:P6}
   \sum_{\varepsilon\in\{\pm1\}} \Normc{w_\varepsilon}{W^m_{(k+\varepsilon)}(\Omega)}
   \cong
   \sum_{\ell=0}^{k} \DNormc{\partial_{r}^{m+1-\ell}\left(\tfrac{1}{r}\right)^{\ell} w }{L^2_1(\Omega)} +
   \DNormc{\partial_{r}^{m-k-1} \left(\tfrac1r\partial_r\right)\left(\tfrac{1}{r}\right)^{k} w }{L^2_1(\Omega)}
\end{equation}
Let us consider now the $X$-seminorms of $w_\varepsilon$:
\begin{equation}
\label{eq:P7}
   \Normc{w_\varepsilon}{X^m_{(k+\varepsilon)}(\Omega)} =
   \sum_{\ell=1}^{[(m-(k+\varepsilon))/2]}
   \DNormc{\partial_{r}^{m-(k+\varepsilon)-2\ell}\left(\tfrac{1}{r}\partial_{r}\right)^{\ell}
   \left(\tfrac{1}{r}\right)^{k+\varepsilon} w_\varepsilon }{L^2_1(\Omega)}
\end{equation}
The contribution of $w_+$ is simple to handle thanks to the identity \eqref{eq:w+}:
\begin{align}
   \DNormc{\partial_{r}^{m-(k+1)-2\ell}\left(\tfrac{1}{r}\partial_{r}\right)^{\ell}
   \left(\tfrac{1}{r}\right)^{k+1} w_+ }{L^2_1(\Omega)} &=
   \DNormc{\partial_{r}^{m-(k+1)-2\ell}\left(\tfrac{1}{r}\partial_{r}\right)^{\ell}
   \left(\tfrac{1}{r}\right) \partial_r \left(\tfrac{1}{r}\right)^{k} w }{L^2_1(\Omega)} \nonumber\\
   & =
   \DNormc{\partial_{r}^{m+1-k-2(\ell+1)}\left(\tfrac{1}{r}\partial_{r}\right)^{\ell+1}
   \left(\tfrac{1}{r}\right)^{k} w }{L^2_1(\Omega)}
\label{eq:P8}
\end{align}
As for $w_-$ we have, instead of \eqref{eq:w+}
\begin{equation}
\label{eq:w-}
   \left(\tfrac{1}{r}\right)^{k-1} w_- = \left(\tfrac{1}{r}\right)^{k-1} (\partial_r w + \tfrac{k}{r} w)
   = \partial_r \left(\tfrac{1}{r}\right)^{k-1} w + (2k-1) \left(\tfrac{1}{r}\right)^{k} w
\end{equation}
Hence
\begin{equation}
\label{eq:P9}
   \left(\tfrac{1}{r}\partial_{r}\right)^{\ell}
   \left(\tfrac{1}{r}\right)^{k-1} w_-  =
   \left(\tfrac{1}{r}\partial_{r}\right)^{\ell}
   \partial_r \left(\tfrac{1}{r}\right)^{k-1} w +
   (2k-1) \left(\tfrac{1}{r}\partial_{r}\right)^{\ell}
   \left(\tfrac{1}{r}\right)^{k} w
\end{equation}
At this point we use the following lemma.

\begin{lemma}
\label{lem:drr}
Let $\ell$ be a positive integer. There holds
\begin{equation}
\label{eq:drr}
   \left(\tfrac{1}{r}\partial_{r}\right)^{\ell} \partial_r (ru) =
   \partial_r^2 \left(\tfrac{1}{r}\partial_{r}\right)^{\ell-1} u +
   2\ell \left(\tfrac{1}{r}\partial_{r}\right)^{\ell} u\,.
\end{equation}
\end{lemma}

\begin{proof}
The identity $\partial_r(ru) = r\partial_ru+u$ implies
\begin{equation}
\label{eq:drr1}
   \left(\tfrac{1}{r}\partial_{r}\right) \partial_r (ru) =
   \partial_r^2  u +
   2 \left(\tfrac{1}{r}\partial_{r}\right) u\,,
\end{equation}
which is \eqref{eq:drr} for $\ell=1$. Assuming that \eqref{eq:drr} holds for $\ell$ we write
\[
\begin{aligned}
   \left(\tfrac{1}{r}\partial_{r}\right)^{\ell+1} \partial_r (ru) &=
   \left(\tfrac{1}{r}\partial_{r}\right) \left(\tfrac{1}{r}\partial_{r}\right)^{\ell} \partial_r (ru)
   \\ &=
   \left(\tfrac{1}{r}\partial_{r}\right) \partial_r^2 \left(\tfrac{1}{r}\partial_{r}\right)^{\ell-1} u +
   2\ell \left(\tfrac{1}{r}\partial_{r}\right)^{\ell+1} u
   \\ &=
   \left(\tfrac{1}{r}\partial_{r}\right) \partial_r \,r \left(\left(\tfrac{1}{r}\partial_{r}\right)^{\ell} u\right) +
   2\ell \left(\tfrac{1}{r}\partial_{r}\right)^{\ell+1} u
\end{aligned}
\]
Using \eqref{eq:drr1} with $v=\left(\tfrac{1}{r}\partial_{r}\right)^{\ell} \!u$ yields formula \eqref{eq:drr} for $\ell+1$, which ends the proof.
\end{proof}

Coming back to identity \eqref{eq:P9} we apply the lemma with $u=\left(\tfrac{1}{r}\right)^{k} w$ and find
\[
   \left(\tfrac{1}{r}\partial_{r}\right)^{\ell}
   \left(\tfrac{1}{r}\right)^{k-1} w_-  =
   \partial_r^2 \left(\tfrac{1}{r}\partial_{r}\right)^{\ell-1} \left(\tfrac{1}{r}\right)^{k} w +
   \tau\dd{\ell,k} \left(\tfrac{1}{r}\partial_{r}\right)^{\ell} \left(\tfrac{1}{r}\right)^{k} w
\]
where we have set $\tau\dd{\ell,k}=2\ell+2k-1$,
which allows to handle the $X$-seminorm of $w_-$
\begin{align*}
   \DNormc{\partial_{r}^{m-(k-1)-2\ell}\left(\tfrac{1}{r}\partial_{r}\right)^{\ell}
   \left(\tfrac{1}{r}\right)^{k-1} w_- }{L^2_1(\Omega)}\\
   &\hskip-8em =
   \DNormc{\partial_{r}^{m-(k-1)-2\ell} \left[
   \partial_r^2 \left(\tfrac{1}{r}\partial_{r}\right)^{\ell-1} \left(\tfrac{1}{r}\right)^{k} w +
   \tau\dd{\ell,k} \left(\tfrac{1}{r}\partial_{r}\right)^{\ell} \left(\tfrac{1}{r}\right)^{k} w\right]
   }{L^2_1(\Omega)} \\
   &\hskip-8em =
   \DNormc{\partial_{r}^{m-(k-1)-2(\ell-1)}
   \left(\tfrac{1}{r}\partial_{r}\right)^{\ell-1} \left(\tfrac{1}{r}\right)^{k} w +
   \tau\dd{\ell,k} \partial_{r}^{m-(k-1)-2\ell} \left(\tfrac{1}{r}\partial_{r}\right)^{\ell}
   \left(\tfrac{1}{r}\right)^{k} w
   }{L^2_1(\Omega)}
\end{align*}
Let us combine this identity with \eqref{eq:P7} and \eqref{eq:P8}. We find
\begin{align*}
   \sum_{\varepsilon\in\{\pm1\}} \Normc{w_\varepsilon}{X^m_{(k+\varepsilon)}(\Omega)} &=
   \sum_{\varepsilon\in\{\pm1\}}
   \sum_{\ell=1}^{[(m-(k+\varepsilon))/2]}
   \DNormc{\partial_{r}^{m-(k+\varepsilon)-2\ell}\left(\tfrac{1}{r}\partial_{r}\right)^{\ell}
   \left(\tfrac{1}{r}\right)^{k+\varepsilon} w_\varepsilon }{L^2_1(\Omega)}
   \\ &\hskip-5em =
   \sum_{\ell=1}^{[(m-k-1)/2]}
   \DNormc{\partial_{r}^{m+1-k-2(\ell+1)}\left(\tfrac{1}{r}\partial_{r}\right)^{\ell+1}
   \left(\tfrac{1}{r}\right)^{k} w }{L^2_1(\Omega)}
   \\ &\hskip-4em +
   \sum_{\ell=1}^{[(m-k+1)/2]}
   \DNormc{\partial_{r}^{m-(k-1)-2(\ell-1)}
   \left(\tfrac{1}{r}\partial_{r}\right)^{\ell-1} \left(\tfrac{1}{r}\right)^{k} w +
   \tau\dd{\ell,k} \partial_{r}^{m-(k-1)-2\ell} \left(\tfrac{1}{r}\partial_{r}\right)^{\ell}
   \left(\tfrac{1}{r}\right)^{k} w
   }{L^2_1(\Omega)}
   \\ &\hskip-5em =
   \sum_{\ell=2}^{[(m+1-k)/2]}
   \DNormc{\partial_{r}^{m+1-k-2\ell}\left(\tfrac{1}{r}\partial_{r}\right)^{\ell}
   \left(\tfrac{1}{r}\right)^{k} w }{L^2_1(\Omega)}
   \\ &\hskip-4em +
   \sum_{\ell=1}^{[(m+1-k)/2]}
   \DNormc{\partial_{r}^{m+1-k-2(\ell-1)}
   \left(\tfrac{1}{r}\partial_{r}\right)^{\ell-1} \left(\tfrac{1}{r}\right)^{k} w +
   \tau\dd{\ell,k} \partial_{r}^{m+1-k-2\ell} \left(\tfrac{1}{r}\partial_{r}\right)^{\ell}
   \left(\tfrac{1}{r}\right)^{k} w
   }{L^2_1(\Omega)}
\end{align*}
Recall that, according to \eqref{eq:4} we have to estimate the sum:
\begin{equation}
\label{eq:Sigma}
   \Sigma :=\sum_{\varepsilon\in\{\pm1\}} \Normc{w_\varepsilon}{W^m_{(k+\varepsilon)}(\Omega)} +
   \sum_{\varepsilon\in\{\pm1\}} \Normc{w_\varepsilon}{X^m_{(k+\varepsilon)}(\Omega)}.
\end{equation}
Combining the latter identity for the $X$-seminorm with the equivalence \eqref{eq:P6} for the $W$-seminorm, we find
\begin{align*}
   \Sigma &\cong
      \sum_{\ell=0}^{k} \DNormc{\partial_{r}^{m+1-\ell}\left(\tfrac{1}{r}\right)^{\ell} w }{L^2_1(\Omega)} +
   \DNormc{\partial_{r}^{m-k-1} \left(\tfrac1r\partial_r\right)\left(\tfrac{1}{r}\right)^{k} w }{L^2_1(\Omega)}
   \\ &\quad +
   \sum_{\ell=2}^{[(m+1-k)/2]}
   \DNormc{\partial_{r}^{m+1-k-2\ell}\left(\tfrac{1}{r}\partial_{r}\right)^{\ell}
   \left(\tfrac{1}{r}\right)^{k} w }{L^2_1(\Omega)}
   \\ &\quad +
   \sum_{\ell=1}^{[(m+1-k)/2]}
   \DNormc{\partial_{r}^{m+1-k-2(\ell-1)}
   \left(\tfrac{1}{r}\partial_{r}\right)^{\ell-1} \left(\tfrac{1}{r}\right)^{k} w +
   \tau\dd{\ell,k} \partial_{r}^{m+1-k-2\ell} \left(\tfrac{1}{r}\partial_{r}\right)^{\ell}
   \left(\tfrac{1}{r}\right)^{k} w
   }{L^2_1(\Omega)}
   \\ &\cong
   \sum_{\ell=0}^{k} \DNormc{\partial_{r}^{m+1-\ell}\left(\tfrac{1}{r}\right)^{\ell} w }{L^2_1(\Omega)} +
   \sum_{\ell=1}^{[(m+1-k)/2]}
   \DNormc{\partial_{r}^{m+1-k-2\ell}\left(\tfrac{1}{r}\partial_{r}\right)^{\ell}
   \left(\tfrac{1}{r}\right)^{k} w }{L^2_1(\Omega)}
   \\ &\quad +
   \sum_{\ell=1}^{[(m+1-k)/2]}
   \DNormc{\partial_{r}^{m+1-k-2(\ell-1)}
   \left(\tfrac{1}{r}\partial_{r}\right)^{\ell-1} \left(\tfrac{1}{r}\right)^{k} w +
   \tau\dd{\ell,k} \partial_{r}^{m+1-k-2\ell} \left(\tfrac{1}{r}\partial_{r}\right)^{\ell}
   \left(\tfrac{1}{r}\right)^{k} w
   }{L^2_1(\Omega)}
\end{align*}
Using Lemma \ref{lem:ineq} we obtain that
\[
   \Sigma \cong
   \sum_{\ell=0}^{k} \DNormc{\partial_{r}^{m+1-\ell}\left(\tfrac{1}{r}\right)^{\ell} w }{L^2_1(\Omega)} +
   \sum_{\ell=0}^{[(m+1-k)/2]}
   \DNormc{\partial_{r}^{m+1-k-2\ell}\left(\tfrac{1}{r}\partial_{r}\right)^{\ell}
   \left(\tfrac{1}{r}\right)^{k} w }{L^2_1(\Omega)}
\]
Since the contribution of $\ell=0$ to the second sum coincides with that of $\ell=k$ in the first sum, we have finally found that
\[
   \sum_{\varepsilon\in\{\pm1\}} \Normc{w_\varepsilon}{W^m_{(k+\varepsilon)}(\Omega)} +
   \sum_{\varepsilon\in\{\pm1\}} \Normc{w_\varepsilon}{X^m_{(k+\varepsilon)}(\Omega)} \cong
   \Normc{w}{W^{m+1}_{(k)}(\Omega)} + \Normc{w}{X^{m+1}_{(k)}(\Omega)}
\]
which is the desired result \eqref{eq:4}.

\subsubsection*{\textbf{Part IV}: Case $k = 0$}
Now $w_-=w_+=\partial_rw$. Hence, still using abbreviation \eqref{eq:Sigma}:
\[
   \sum_{\varepsilon\in\{\pm1\}} \Normc{w_\varepsilon}{W^m_{(k+\varepsilon)}(\Omega)} +
   \sum_{\varepsilon\in\{\pm1\}} \Normc{w_\varepsilon}{X^m_{(k+\varepsilon)}(\Omega)} =: \Sigma =
   2 \left(\Normc{\partial_rw}{W^m_{(1)}(\Omega)} +
   \Normc{\partial_rw}{X^m_{(1)}(\Omega)} \right)
\]
Using the definitions of the $W$ and $X$-seminorms, we immediately find (note that $m\ge1$)
\begin{align*}
   \Sigma \ & = \
   \sum_{\ell=0}^{1} \DNormc{\partial_{r}^{m-\ell} \left(\tfrac{1}{r}\right)^{\ell}\partial_r w }{L^2_1(\Omega)} +
   \sum_{\ell=1}^{[(m-1)/2]}
   \DNormc{\partial_{r}^{m-1-2\ell}\left(\tfrac{1}{r}\partial_{r}\right)^{\ell}
   \left(\tfrac{1}{r}\right) \partial_r w }{L^2_1(\Omega)}
   \\ & = \
   \DNormc{\partial_{r}^{m+1} w }{L^2_1(\Omega)} +
   \DNormc{\partial_{r}^{m-1}\left(\tfrac{1}{r}\partial_r\right)  w }{L^2_1(\Omega)} +
   \sum_{\ell=1}^{[(m-1)/2]}
   \DNormc{\partial_{r}^{m-1-2\ell}\left(\tfrac{1}{r}\partial_{r}\right)^{\ell+1} w }{L^2_1(\Omega)}
   \\ & = \
   \DNormc{\partial_{r}^{m+1} w }{L^2_1(\Omega)} +
   \DNormc{\partial_{r}^{m-1}\left(\tfrac{1}{r}\partial_r\right)  w }{L^2_1(\Omega)}
    +
   \sum_{\ell=2}^{[(m+1)/2]}
   \DNormc{\partial_{r}^{m+1-2\ell}\left(\tfrac{1}{r}\partial_{r}\right)^{\ell} w }{L^2_1(\Omega)}
\end{align*}
Hence
\[
   \Sigma =\DNormc{\partial_{r}^{m+1} w }{L^2_1(\Omega)} +
   \sum_{\ell=1}^{[(m+1)/2]}
   \DNormc{\partial_{r}^{m+1-2\ell}\left(\tfrac{1}{r}\partial_{r}\right)^{\ell} w }{L^2_1(\Omega)}
    = 
   \DNormc{ w }{W^{m+1}_{(0)}(\Omega)} +
   \DNormc{ w }{X^{m+1}_{(0)}(\Omega)}
\]
which is the desired result \eqref{eq:4}.
\medskip

The proof of Theorem \ref{th:semi} is now complete.


\section{Comparison of results, traces on the axis of rotation}
\label{s:PF}

\subsection{The assumptions and characterizations given in \cite{BDMbook}}
In the book \cite{BDMbook}, the meridian domain $\Omega$ is assumed to be a polygonal domain (that is, a Lipschitz domain whose boundary is a finite union of segments) with the supplementary condition that the intersection $\Gamma_0$ of $\partial\Omega$ with the rotation axis $\cA$ is a union of segments with non empty interior. This excludes isolated points from this intersection and implies that $\breve\Omega$ is itself a Lipschitz domain. 

The characterization of $H^m_{(k)}(\Omega)$ of \cite{BDMbook} is given in Theorem II.3.1, using some subspaces $H^m_+(\Omega)$ and $H^m_-(\Omega)$ of the Sobolev space $H^m_1(\Omega)$ with measure $r\,\rd r\rd z$. We find that a reformulation using different subspaces could make the result of Theorem II.3.1 easier to understand.

\begin{notation}
\label{n:BDM1}
Let $m\in\N$.

\noindent{\em (i)} Recall that 
\[
   H^m_1(\Omega) = \Big\{w\in L^2_1(\Omega),\quad 
  \sum_{|\alpha|\le m} 
  \DNormc{\partial^{\alpha}w}{L^{2}_{1}(\Omega)}<\infty \Big\}
\]
{\em (ii)} Introduce for $m\ge1$ the non-closed subspace of $H^m_1(\Omega)$
\[
   H^m_\strong(\Omega) = \Big\{w\in H^m_1(\Omega),\quad 
  \DNormc{\tfrac{1}{r}\,\partial^{m-1}_r w}{L^{2}_{1}(\Omega)}<\infty \Big\}
\]
{\em (iii)} Denote the corresponding Kondrat'ev type space by $V^m_1(\Omega)$,
\[
   V^m_1(\Omega) = \Big\{w\in H^m_1(\Omega),\quad 
  \sum_{|\alpha|\le m} \DNormc{(\tfrac{1}{r})^{m-|\alpha|}\,
  \partial^{\alpha} w}{L^{2}_{1}(\Omega)}<\infty \Big\}
\]
Natural associated norms are denoted by $\DNorm{w}{H^m_1(\Omega)}$, $\DNorm{w}{H^m_\strong(\Omega)}$, and $\DNorm{w}{V^m_1(\Omega)}$, respectively.
\end{notation}

Recall
\[
   \Gamma_0 = \cA\cap\partial\Omega.
\]
With the assumption that $\Gamma_0$ is a finite union of intervals, one finds that the trace operators
\[
   H^m_1(\Omega)\ni v\longmapsto \partial^j_r v\on{\Gamma_0}
\]
make sense if $0\le j\le m-2$ and are continuous from $H^m_1(\Omega)$ into $H^{m-1-j}(\Gamma_0)$, as follows for example from \cite[Theorem 3.6.1]{Triebel1978}.  This allows to define the following subspaces of $H^m_1(\Omega)$ and $H^m_\strong(\Omega)$:

\begin{notation}
\label{n:BDM2}
Let $m\in\N$, $m\ge1$. Set
\[
   T^m_1(\Omega) = \Big\{u\in H^m_1(\Omega),\quad 
   \partial^{m-2\ell}_r v\on{\Gamma_0}=0,\quad \ell=1,\ldots [\tfrac{m}{2}] \Big\}
\]
and
\[
   T^m_\strong(\Omega) = \Big\{u\in H^m_\strong(\Omega),\quad 
   \partial^{m-1-2\ell}_r v\on{\Gamma_0}=0,\quad \ell=1,\ldots [\tfrac{m-1}{2}]\Big\}
\]
\end{notation}
This means that the meaningless trace condition $\partial^{m-1}_r v\on{\Gamma_0}=0$ is replaced by the finiteness of the weighted norm $\DNorm{\tfrac{1}{r}\,\partial^{m-1}_r v}{L^{2}_{1}(\Omega)}$. Finally, for $k\in\Z$, we define
\begin{equation}
\label{eq:Zk}
   Z^k(\Omega) = \Big\{u\in H^{|k|+1}_1(\Omega),\quad 
   \partial^{j}_r v\on{\Gamma_0}=0,\quad j=0,\ldots |k|-1\Big\}
\end{equation}
Note that, when $|k|\le m-1$, the space $Z^k(\Omega) \cap T^m_1(\Omega)$ is closed in $H^m_1(\Omega)$ and $Z^k(\Omega) \cap T^m_\strong(\Omega)$ is closed in $H^m_\strong(\Omega)$.

Then \cite[Theorem II.3.1]{BDMbook} can be reformulated (for integer Sobolev exponents) as follows

\begin{theorem}
\label{t:BDM}
Let $m\in\N$ and $k\in\Z$. Let us introduce the spaces
\begin{equation}
\label{e:spBmk}
   B^m_{(k)}(\Omega) := 
   \begin{cases}
   V^m_1(\Omega)  & \mbox{if}\quad |k|\ge m \\[0.5ex]
   Z^k(\Omega) \cap T^m_1(\Omega) \quad 
   & \mbox{if}\quad |k|\le m-1 \;\;\&\;\; m-k\not\in 2\Z \\[0.5ex]
   Z^k(\Omega) \cap T^m_\strong(\Omega) \quad 
   & \mbox{if}\quad |k|\le m-1 \;\;\&\;\;  m-k\in 2\Z \\
   \end{cases}
\end{equation}
with their associated norms
\begin{equation}
\label{e:Bmk}
   \DNormc{w}{B^m_{(k)}(\Omega)} = 
   \begin{cases}
   \DNormc{w}{H^m_1(\Omega)} + \DNormc{\left(\tfrac{k}{r}\right)^m w}{L^2_1(\Omega)}  & 
   \mbox{if}\quad |k|\ge m \\[0.5ex]
   \DNormc{w}{H^m_1(\Omega)} \quad & 
   \mbox{if}\quad |k|\le m-1 \;\;\&\;\; m-k\not\in 2\Z \\[0.5ex]
   \DNormc{w}{H^m_\strong(\Omega)} \quad & 
   \mbox{if}\quad |k|\le m-1 \;\;\&\;\; m-k\in 2\Z \\
   \end{cases}
\end{equation}
Then 
\begin{equation}
\label{e:ZT}
   H^m_{(k)}(\Omega) = 
   B^m_{(k)}(\Omega)
\end{equation}
with equivalence of norms in the sense that
there exists a constant $\beta_{m,\Omega}\ge1$ depending on $m$ and $\Omega$, but not on $k$, such that 
\begin{equation}
\label{e:normeqB}
   \beta_{m,\Omega}^{-1} \DNorm{w}{H^m_{(k)}(\Omega)} \le
   \DNorm{w}{B^m_{(k)}(\Omega)} \le
   \beta_{m,\Omega} \DNorm{w}{H^m_{(k)}(\Omega)}
\end{equation}
\end{theorem}

In fact, the proof of \cite[Theorem II.3.1]{BDMbook} is performed there only for {\em cylinders}, i.e., when $\Omega$ is a rectangle\footnote{The proofs of \cite{BDMbook} rely on the assumption that the domain $\breve\Omega$ is locally diffeomorphic to  cylinders. It should be noted that smooth diffeomorphisms $(r,z)\mapsto(r',z')$ that preserve axisymmetry are constrained along the axis $r=0$ by the condition $\partial_rz'=0$ on $r=0$. In particular, the opening angles of $\Omega$ along the rotation axis cannot be modified.
}. So there remains a gap to generalize the result to any polygonal $\Omega$. We fill this gap in the Appendix (Theorem \ref{t:BDMext}), proving that the spaces 
$B^m_{(k)}(\Omega)$ defined in \eqref{e:spBmk}
have suitable extension properties (Lemma \ref{l:exten}).

\subsection{Comparison}
Theorem \ref{t:main} and Theorem \ref{t:BDM} give different answers to the same question. For the ease of comparison, let us assume that $\breve\Omega$ is a cylinder or an annulus. More specifically, choose $R>0$, an open bounded interval $\cI\subset\cA$, and set for any $\varepsilon\in[0,R)$
\[
   \Omega_\varepsilon = (\varepsilon,R)\times \cI\quad\mbox{and}\quad
   \breve\Omega_\varepsilon = \cD(\varepsilon,R)\times\cI
\]
where $\cD(\varepsilon,R)$ is the annulus of radii $\varepsilon$ and $R$ if $\varepsilon>0$, and the disc of radius $R$ otherwise.

As a consequence of the two theorems \ref{t:main} and \ref{t:BDM}, the norms $\DNorm{\cdot}{C^m_{(k)}(\Omega_\varepsilon)}$, cf \eqref{e:Cmk}, and $\DNorm{\cdot}{B^m_{(k)}(\Omega_\varepsilon)}$ cf \eqref{e:Bmk}, are equivalent with equivalence constants {\em independent of $k$}. But, as we will see now, these constants are not uniform with respect to the domain, namely with respect to $\varepsilon$ in the present case.

Consider functions that are polynomial in $r$ in the meridian domain $\Omega_\varepsilon$ and investigate their $W$- and $X$-seminorms. 
Let $n\in\N$ be a natural integer and set
\[
   w(r,z) = r^n.
\]
We are going to characterize the integers $n$ such that
\begin{itemize}
\item[\emph{(i)}]  $\Norm{w}{W^m_{(k)}(\Omega_\varepsilon)}$ and $\Norm{w}{X^m_{(k)}(\Omega_\varepsilon)}$ remain uniformly bounded when $\varepsilon\searrow 0$ (keeping $\varepsilon>0$).
\item[\emph{(ii)}]  $\Norm{w}{W^m_{(k)}(\Omega_{0})}$ and $\Norm{w}{X^m_{(k)}(\Omega_{0})}$ are finite for $\Omega_{0}=(0,1)$.
\end{itemize}
Straightforward calculations yield:

\begin{lemma}
\label{lem:pol}
Let $k\in\Z$, $m\in\N$, and $n\in\N$. Set $w(r)=r^n$.

\noindent
(i) The operators involved in $\Normc{w}{W^m_{(k)}(\Omega_\varepsilon)}$ are $\partial_{r}^{m-\ell}\left(\tfrac{1}{r}\right)^{\ell}$ for $\ell=0,\ldots, \min(|k|,m)$. We have
\begin{equation}
\label{eq:Wpol}
   \partial_{r}^{m-\ell}\left(\tfrac{1}{r}\right)^{\ell} w = P^W_{m,n\,;\,\ell} \,r^{n-m}\quad\mbox{with}\quad
   P^W_{m,n\,;\,\ell} = \prod_{p=\ell}^{m-1} (n-p).
\end{equation}
(ii) The operators involved in  $\Normc{w}{X^m_{(k)}(\Omega_\varepsilon)}$ are $\partial_{r}^{m-|k|-2\ell}\left(\tfrac{1}{r}\partial_{r}\right)^{\ell} \left(\tfrac{1}{r}\right)^{|k|}$ for $1\le\ell\le (m-|k|)/2$. We have
\begin{multline}
\label{eq:Xpol}
   \qquad\partial_{r}^{m-|k|-2\ell}\left(\tfrac{1}{r}\partial_{r}\right)^{\ell}
      \left(\tfrac{1}{r}\right)^{|k|} w = P^X_{m,k,n\,;\,\ell} \;r^{n-m}\\ \mbox{with}\quad
   P^X_{m,k,n\,;\,\ell} = \prod_{p=|k|+2\ell}^{m-1} (n-p) \; \prod_{q=0}^{\ell-1} (n-|k|-2q) .\qquad
\end{multline}
\end{lemma}

We note the following behaviors of the $L^2_1$-norm of the function $r^{n-m}$ on the interval $(\varepsilon,1)$ as $\varepsilon$ tends to $0$:
\begin{equation}
\label{eq:r0}
   \DNormc{r^{n-m}}{L^2_1(\Omega_\varepsilon)} =
   \begin{cases}
   \cO(1) & \mbox{if}\quad n-m>-1\\
   \cO(\log\frac{1}{\varepsilon}) & \mbox{if}\quad n-m=-1\\
   \cO(\varepsilon^{2(n-m+1)}) & \mbox{if}\quad n-m<-1\\
\end{cases} \qquad \mbox{as} \quad \varepsilon\to 0.
\end{equation}
Hence seminorms $\Norm{w}{W^m_{(k)}(\Omega_\varepsilon)}$ and $\Norm{w}{X^m_{(k)}(\Omega_\varepsilon)}$ are finite for all $\varepsilon\ge0$ when $n\ge m$.

Consider now the case when $n\le m-1$.

In view of \eqref{eq:Wpol} and \eqref{eq:r0}, we see that the seminorm $\Norm{w}{W^m_{(k)}(\Omega_\varepsilon)}$ is uniformly bounded as $\varepsilon\to0$ if and only if all constants $P^W_{m,n\,;\,\ell}$ are zero for $\ell=0,\ldots, \min(|k|,m)$. This happens if and only if
\[
   P^W_{m,n\,;\,\min(|k|,m)} = 0, \qquad\mbox{i.e.}\qquad
   \prod_{p=\min(|k|,m)}^{m-1} (n-p) = 0,
\]
i.e. when $|k|\le n \le m-1$.

Likewise, owing to \eqref{eq:Xpol}, we find that the seminorm $\Norm{w}{X^m_{(k)}(\Omega_\varepsilon)}$ is uniformly bounded as $\varepsilon\to0$ if and only if all constants $P^X_{m,k,n\,;\,\ell}$ are zero for $1\le\ell\le (m-|k|)/2$. A necessary and sufficient condition for this is:
\[
   n \in |k|+2\N.
\]
We have obtained:

\begin{lemma}
\label{lem:pol2}
Let $k\in\Z$, $m\in\N$, and $n\in\N$. Set $w(r,z)=r^n$. The following conditions are equivalent:
\begin{itemize}
\item[{(i)}] Seminorms $\Norm{w}{W^m_{(k)}(\Omega_\varepsilon)}$ and $\Norm{w}{X^m_{(k)}(\Omega_\varepsilon)}$ remain uniformly bounded when $\varepsilon\searrow 0$,
\item[{(ii)}] Seminorms $\Norm{w}{W^m_{(k)}(\Omega_{0})}$ and $\Norm{w}{X^m_{(k)}(\Omega_{0})}$ are finite,
\item[{(iii)}] $n\ge m$ \ or \ $n - |k| \in 2\N$.
\end{itemize}
\end{lemma}

Using the characterization of traces of $H^m_1(\Omega_0)$ on the interval $\cI$, we deduce

\begin{lemma}
\label{lem:tra}
Let $k\in\Z$, $m\in\N$. Choose $w\in C^m_{(k)}(\Omega_0)$. Then the traces $\partial^j_r w\on{\cI}$ on the axis are zero for all $j$ in the set
\begin{equation}
\label{e:tra}
   \N_{k,m} := \{0,\ldots,|k|-1\} \;\cup\;
   \{|k|+1+2\ell,\quad \forall\ell\in\N \ \ \mbox{such that}\ \ |k|+1+2\ell<m-1\}
\end{equation}
\end{lemma}

These two lemmas give the resemblance and difference between the characterizations of $H^m_{(k)}(\Omega)$ proved here and those of \cite{BDMbook}:
\begin{enumerate}
\item The trace conditions $\partial^j_r w\on{\cI}=0$ for $j\in\N_{k,m}$ are the same as those invoked in \cite[Theorem II.3.1]{BDMbook}, see Theorem \ref{t:BDM}.
\item For $w$ defined as $w(r,z)=r^j$ with $j\in\N_{k,m}$, we have
\[
   w\not\in B^m_{(k)}(\Omega_0), \ w\not\in C^m_{(k)}(\Omega_0)
   \quad\mbox{and}\quad
   w\in B^m_{(k)}(\Omega_\varepsilon)\cap C^m_{(k)}(\Omega_\varepsilon),\;\varepsilon>0
\] 
with {\em bounded norms} in $B^m_{(k)}(\Omega_\varepsilon)$ as $\varepsilon\to0$, whereas the norms in $C^m_{(k)}(\Omega_\varepsilon)$ blow up as $\varepsilon\to0$: In contrast with the norm $C^m_{(k)}(\Omega_\varepsilon)$, there is no awareness of trace conditions in the norm $B^m_{(k)}(\Omega_\varepsilon)$.
\end{enumerate}


\section{Vector fields}
\label{s:vect}
In order to keep the exposition as short as possible, let us assume that $\breve \Omega$ is bidimensional, so that vector fields $\breve\bu: (x,y)\mapsto \breve\bu(x,y)$ have 2 components
\[
   \breve\bu =  u_x\,\be_x +  u_y\, \be_y
\]
where $\be_x$ and $\be_y$ form the canonical base in $\R^2$ (adding the third variable and a third component will be a simple exercise). The radial and angular components $u_r$ and $u_\theta$ of $\bu$ are defined as
\[
   u_r = u_x\cos\theta + u_y\sin\theta\quad\mbox{and}\quad
   u_\theta = -u_x\sin\theta + u_y\cos\theta.
\]
Note that $\breve\bu = u_r\be_r + u_\theta\be_\theta$ with the canonical radial and angular unit vectors $\be_r=\cos\theta\,\be_x+\sin\theta\,\be_y$ and $\be_\theta=-\sin\theta\,\be_x+\cos\theta\,\be_y$. The axisymmetric systems which motivate our study such as Lam\'e, Stokes, Maxwell, have coefficients {\em independent of $\theta$} if they are formulated in polar (or cylindrical) components.

The Fourier expansion of vector fields takes the form
\begin{equation}
\label{e:FSV}
   \breve\bu(\bx) = 
   \sum_{k\in\Z} \big(u^{k}_r(r) \,\be_r + u^k_\theta(r)\be_\theta\big)
   \, e^{ik\theta} =: \sum_{k\in\Z} \breve\bu^k(\bx)
\end{equation}
where $u^k_r$ and $u^k_\theta$ are the classical Fourier coefficients of $u_r$ and $u_\theta$ in the sense of \eqref{eq:FouCo}.
For any natural integer $m$, define the vector $H^{m}_{(k)}$-norm by
\[
   \DNorm{(w_r,w_\theta)}{\bH^{m}_{(k)}(\Omega)} :=
   \DNorm{(w_r \,\be_r + w_\theta \,\be_\theta)\, e^{ik\theta}}{\bH^m(\breve\Omega)}
\]
where $\bH^m(\breve\Omega)$ is the vector Sobolev space $H^{m}(\breve\Omega)^2$. 
Here follows the vector analogue of the direct sum identity \eqref{eq:dirsum2}.

\begin{proposition}
Let $m\in\Z$. There holds
\begin{equation}
\label{eq:dirsumV}
   \DNormc{\breve\bu}{\bH^m(\breve\Omega)} =
   \sum_{k\in\Z} \DNormc{(u^k_r,u^k_\theta)}{\bH^m_{(k)}(\Omega)}.
\end{equation}
Moreover we have the following relation linking vector and scalar $H^m_{(k)}$-norms:
\begin{equation}
\label{eq:ukV}
   \DNormc{(u^k_r,u^k_\theta)}{\bH^m_{(k)}(\Omega)} =
   \tfrac12 \DNormc{u^k_r + i\ee u^k_\theta}{H^m_{(k+1)}(\Omega)} +
   \tfrac12 \DNormc{u^k_r - i\ee u^k_\theta}{H^m_{(k-1)}(\Omega)} .
\end{equation}
\end{proposition}

\begin{proof}
Elementary calculations yield
\[
\begin{aligned}
   (u_x \,\be_x + u_y \,\be_y) & =
   \tfrac12(u_x + i\ee u_y)(\be_x-i\,\be_y) 
   + \tfrac12(u_x - i\ee u_y)(\be_x+i\,\be_y) \\
   = (u_r \,\be_r + u_\theta \,\be_\theta) &=
   \tfrac12(u_r + i\ee u_\theta)(\be_r-i\,\be_\theta) 
   + \tfrac12(u_r - i\ee u_\theta)(\be_r+i\,\be_\theta) \\
&=
   \tfrac12(u_r + i\ee u_\theta) e^{i\theta} (\be_x-i\,\be_y) 
   + \tfrac12(u_r - i\ee u_\theta) e^{-i\theta} (\be_x+i\,\be_y) \\
\end{aligned}
\]
Hence
\[
   \DNormc{\breve\bu}{\bH^m(\breve\Omega)} =
   \tfrac12 \DNormc{u_x + i\ee u_y}{H^m(\breve\Omega)} +
   \tfrac12 \DNormc{u_x - i\ee u_y}{H^m(\breve\Omega)} 
\]
and for all $k\in\Z$ we have the identity between scalar Fourier coefficients
\begin{equation}
\label{e:FouCoV}
   u^k_x + i\ee u^k_y = u^{k-1}_r + i\ee u^{k-1}_\theta
   \quad\mbox{and}\quad
   u^k_x - i\ee u^k_y = u^{k+1}_r - i\ee u^{k+1}_\theta
\end{equation}
From the direct sum identity \eqref{eq:dirsum2} in the scalar case we find
\[
\begin{aligned}
   \DNormc{\breve\bu}{\bH^m(\breve\Omega)} &=
   \tfrac12 \sum_{k\in\Z}   
   \DNormc{u^k_x + i\ee u^k_y}{H^m_{(k)}(\Omega)} +
   \DNormc{u^k_x - i\ee u^k_y}{H^m_{(k)}(\Omega)} \\ &=
   \tfrac12 \sum_{k\in\Z}   
   \DNormc{u^{k-1}_r + i\ee u^{k-1}_\theta}{H^m_{(k)}(\Omega)} +
   \DNormc{u^{k+1}_r - i\ee u^{k+1}_\theta}{H^m_{(k)}(\Omega)} \\ &=
   \tfrac12 \sum_{k\in\Z}   
   \DNormc{u^{k}_r + i\ee u^{k}_\theta}{H^m_{(k+1)}(\Omega)} +
   \DNormc{u^{k}_r - i\ee u^{k}_\theta}{H^m_{(k-1)}(\Omega)} .
\end{aligned}
\]
Since
\[
   (u^k_r \,\be_r + u^k_\theta \,\be_\theta) e^{ik\theta} =
   \tfrac12(u^k_r + i\ee u^k_\theta) e^{i(k+1)\theta} (\be_x-i\,\be_y) 
   + \tfrac12(u^k_r - i\ee u^k_\theta) e^{i(k-1)\theta} (\be_x+i\,\be_y) 
\]
from which we find
\[
\begin{aligned}
   \DNormc{(u^k_r \,\be_r + u^k_\theta \,\be_\theta)\, e^{ik\theta}}
   {\bH^m(\breve\Omega)}
   &=
   \tfrac12\DNormc{(u^k_r + i\ee u^k_\theta)\, e^{i(k+1)\theta}}{H^m(\breve\Omega)} +
   \tfrac12\DNormc{(u^k_r - i\ee u^k_\theta)\, e^{i(k-1)\theta}}{H^m(\breve\Omega)}
   \\
   &=
   \tfrac12\DNormc{u^k_r + i\ee u^k_\theta}{H^m_{(k+1)}(\Omega)} +
   \tfrac12\DNormc{u^k_r - i\ee u^k_\theta}{H^m_{(k-1)}(\Omega)}\,,
\end{aligned}
\]
which completes the proof.
\end{proof}

The main result of this section is the characterization of the vector $H^m_{(k)}$-norms by weighted norms. Recall that according to Theorem \ref{t:main} the scalar $H^m_{(k)}$-norm is  equivalent to the weighted norm $C^m_{(k)}(\Omega)$ of \eqref{e:Cmk}.

\begin{theorem}
\label{t:mainV}
Let $m\in\N$ and $k\in\Z$. Using the weighted norm $C^m_{(k)}(\Omega)$ from \eqref{e:Cmk}, we define the vector weighted norm $\bC^m_{(k)}(\Omega)$ by
\begin{align*}
    \DNormc{(w_r,w_\theta)}{\bC^{m}_{(k)}(\Omega)} =
    \left\{
    \begin{array}{ll}
      \DNormc{(w_r,w_\theta)}{C^{m}_{(k)}(\Omega)\times C^{m}_{(k)}(\Omega)} \quad
      &\mbox{if}\quad  |k| \geq m+1,\\[0.5ex]
      \DNormc{(w_r,w_\theta)}{C^{m}_{(|k|-1)}(\Omega)\times C^{m}_{(|k|-1)}(\Omega)}
      &\!\!\!\! +\; \DNormc{\left(\frac{1}{r}\right)^{|k|} (w_{r} + i \frac{k}{|k|}
      w_{\theta})}{H^{m-|k|}_{1}(\Omega)} \\[-0.75ex]
      &\mbox{if}\quad  1\leq|k|\leq m,\\[0.5ex]
      \DNormc{(w_r,w_\theta)}{C^{m}_{(1)}(\Omega)\times C^{m}_{(1)}(\Omega)}
      &\mbox{if}\quad k=0.
    \end{array}
    \right.
\end{align*}
Then we have the norm equivalence
\begin{equation}
\label{e:normeqV}
   c_m\DNorm{(w_r,w_\theta)}{\bH^m_{(k)}(\Omega)} \le 
   \DNorm{(w_r,w_\theta)}{\bC^m_{(k)}(\Omega)} \le 
   C_m\DNorm{(w_r,w_\theta)}{\bH^m_{(k)}(\Omega)}
\end{equation}
with positive constants $c_m$ and $C_m$ independent of $k$ and $\Omega$.
\end{theorem}

Not surprisingly, the characterization found in \cite[Theorem II.3.6]{BDMbook} displays also extra norms and trace conditions for the term $w_{r} + i \frac{k}{|k|} w_{\theta}$. Note also that the standard contribution to the norm $\DNorm{(w_r,w_\theta)}{\bC^{m}_{(k)}(\Omega)}$ can be written in a unified way as
\[
   \DNorm{(w_r,w_\theta)}{C^{m}_{(|k|-1)}(\Omega)\times C^{m}_{(|k|-1)}(\Omega)}
\]
(since the norms $C^{m}_{(k)}$ and $C^{m}_{(|k|-1)}$ are equivalent when $|k|\ge m$
and, for $k=0$, the norms $C^{m}_{(1)}$ and $C^{m}_{(-1)}$ are the same).

\medskip\noindent{\em Elements of proof.}
The proof of Theorem \ref{t:mainV} is in the same spirit as the proof of Theorem \ref{th:semi} for the scalar case. Assuming first that $k\ge0$, owing to \eqref{eq:ukV} we have to prove that
\[
   \DNormc{w_r + i\ee w_\theta}{H^m_{(k+1)}(\Omega)} +
   \DNormc{w_r - i\ee w_\theta}{H^m_{(k-1)}(\Omega)} 
\]
is equivalent to $\DNormc{(w_r,w_\theta)}{\bC^{m}_{(k)}(\Omega)}$. Using Theorem \ref{t:main}, we are reduced to studying
\[
   \DNormc{w_r + i\ee w_\theta}{C^m_{(k+1)}(\Omega)} +
   \DNormc{w_r - i\ee w_\theta}{C^m_{(k-1)}(\Omega)} .
\]
Then the trick is to examine the pieces of norms in each term, use repeatedly equivalences provided by Lemma~\ref{lem:ineq}, and reassemble identical pairs by the polarization identity
\[
   \DNormc{w_r + i\ee w_\theta}{ } +
   \DNormc{w_r - i\ee w_\theta}{ } = 
   2( \DNormc{w_r}{ } + \DNormc{w_\theta}{ }).
\]
This method gives immediately the result when $k\ge m+1$, since in that case $k+1$ and $k-1$ are $\ge m$, so the norms $C^m_{(k+1)}(\Omega)$ and $C^m_{(k-1)}(\Omega)$ are formed by pairwise equivalent terms. The case $k=0$ is very simple too since $|k+1| = |k-1| = 1$. The remaining cases $1\le k\le m$ require a more careful examination.

\smallskip\noindent
{\em (i)} \ \ $k=m$. \ \ It is straightforward that
\[
   \DNormc{w}{C^m_{(m+1)}(\Omega)} \cong \sum_{\ell=0}^{m} 
   \DNormc{\left(\tfrac{1}{r}\right)^{\ell} w }{H^{m-\ell}_1(\Omega)} \cong
   \DNormc{w}{C^m_{(m-1)}(\Omega)} + 
   \DNormc{\left(\tfrac{1}{r}\right)^m w }{L^2_1(\Omega)}
\]
Thus
\begin{multline*}
   \DNormc{w_r + i\ee w_\theta}{C^m_{(k+1)}(\Omega)} +
   \DNormc{w_r - i\ee w_\theta}{C^m_{(k-1)}(\Omega)} \cong \\
   \DNormc{w_r + i\ee w_\theta}{C^m_{(m-1)}(\Omega)} + 
   \DNormc{w_r - i\ee w_\theta}{C^m_{(m-1)}(\Omega)} +    
   \DNormc{\left(\tfrac{1}{r}\right)^m (w_r + i\ee w_\theta) }{L^2_1(\Omega)},
\end{multline*}
which, after reassembling, gives the desired result.

\smallskip\noindent
{\em (ii)} \ \ $k=m-1$. \ \  The discussion is easier if we introduce the two norms, cf \eqref{eq:Wmk} and \eqref{eq:Xmk}, the sum of which gives back the scalar $C^m_{(k)}$-norm:
\[
\begin{gathered}
   \DNormc{w}{W^m_{(k)}(\Omega)} =
   \sum_{\ell=0}^{\min\{|k|,m\}} 
   \DNormc{\left(\tfrac{1}{r}\right)^{\ell} w }{H^{m-\ell}_1(\Omega)} \\
   \DNormc{w}{X^m_{(k)}(\Omega)} =
   \sum_{\ell=1}^{[(m-|k|)/2]}
   \DNormc{\left(\tfrac{1}{r}\partial_{r}\right)^{\ell}
   \left(\tfrac{1}{r}\right)^{|k|} w }{H^{m-|k|-2\ell}_1(\Omega)} .
\end{gathered}
\]
Then, abridging $w_r \pm i\ee w_\theta$ by $w_\pm$,
\[
\begin{gathered}
   \DNormc{w_+}{C^m_{(k+1)}(\Omega)} =
   \DNormc{w_+}{W^m_{(k-1)}(\Omega)} +
   \DNormc{\left(\tfrac{1}{r}\right)^{m-1} w_+}{H^1_1(\Omega)} +
   \DNormc{\left(\tfrac{1}{r}\right)^m w_+}{L^2_1(\Omega)}
 \\
   \DNormc{w_-}{C^m_{(k-1)}(\Omega)} =
   \DNormc{w_-}{W^m_{(k-1)}(\Omega)} +
   \DNormc{\left(\tfrac{1}{r}\partial_{r}\right)
   \left(\tfrac{1}{r}\right)^{m-1} w_- }{L^2_1(\Omega)} .
\end{gathered}
\]
As $\frac{1}{r}\partial_r v = \partial_r(\frac{1}{r} v) + \frac{1}{r^2}v$, we find the equivalence between $L^2_1(\Omega)$-norms:
\begin{equation}
\label{eq:eqL21}
   \DNorm{\partial_r \left(\tfrac{1}{r}\right)^{m-1} w_+}{} +
   \DNorm{\left(\tfrac{1}{r}\right)^m w_+}{} \;\cong\;
   \DNorm{\partial_r \left(\tfrac{1}{r}\right)^{m-1} w_+}{} +
   \DNorm{(\tfrac{1}{r}\partial_{r})
   \left(\tfrac{1}{r}\right)^{m-2} w_+ }{}
\end{equation}
which allows reassembling and yields the desired result. 

\smallskip\noindent
{\em (iii)} \ \ $1\le k \le m-2$. \ \  
We start by
\begin{subequations}
\begin{align}
   \DNormc{w_+}{C^m_{(k+1)}(\Omega)} &=
   \DNormc{w_+}{W^m_{(k-1)}(\Omega)} +
   \DNormc{\left(\tfrac{1}{r}\right)^{k} w_+}{H^{m-k}_1(\Omega)} \\ 
   \label{e:b}
   &\qquad +
   \DNormc{\left(\tfrac{1}{r}\right)^{k+1} w_+}{H^{m-k-1}_1(\Omega)} +
   \DNormc{w_+}{X^m_{(k+1)}(\Omega)} 
\end{align}
\end{subequations}
Using Lemma \ref{lem:drr}, we find the equivalence of part \eqref{e:b} with
\[
   \DNormc{\left(\tfrac{1}{r}\right)^{k+1} w_+}{H^{m-k-1}_1(\Omega)} +
   \sum_{\ell=1}^{[(m-k-1)/2]}
   \DNormc{\left(\tfrac{1}{r}\partial_{r}\right)^{\ell+1}
   \left(\tfrac{1}{r}\right)^{k-1} w_+ }{H^{m-k-1-2\ell}_1(\Omega)} 
\]
that coincides with
\[
   \DNormc{\left(\tfrac{1}{r}\right)^{k+1} w_+}{H^{m-k-1}_1(\Omega)} +
   \DNormc{w_+}{X^m_{(k-1)}(\Omega)} -    \DNormc{\left(\tfrac{1}{r}\partial_{r}\right)
   \left(\tfrac{1}{r}\right)^{k-1} w_+ }{H^{m-k-1}_1(\Omega)}
\]
Now, using as above in \eqref{eq:eqL21}, the equivalence between $H^{m-k-1}_1(\Omega)$-norms:
\[
   \DNorm{\partial_r \left(\tfrac{1}{r}\right)^{k} w_+}{} +
   \DNorm{\left(\tfrac{1}{r}\right)^{k+1} w_+}{} \;\cong\;
   \DNorm{\partial_r \left(\tfrac{1}{r}\right)^{k} w_+}{} +
   \DNorm{(\tfrac{1}{r}\partial_{r})
   \left(\tfrac{1}{r}\right)^{k-1} w_+ }{}
\]
we obtain the equivalence for the norm $\DNormc{w_+}{C^m_{(k+1)}(\Omega)}$:
\[
\begin{aligned}
   \DNormc{w_+}{C^m_{(k+1)}(\Omega)} &=
   \DNormc{w_+}{W^m_{(k-1)}(\Omega)} +
   \DNormc{\left(\tfrac{1}{r}\right)^{k} w_+}{H^{m-k}_1(\Omega)} 
   + \DNormc{w_+}{X^m_{(k-1)}(\Omega)} \\
   &=
   \DNormc{w_+}{C^m_{(k-1)}(\Omega)} +
   \DNormc{\left(\tfrac{1}{r}\right)^{k} w_+}{H^{m-k}_1(\Omega)} 
\end{aligned}
\]
After reassembling with the norm $C^m_{(k-1)}(\Omega)$ of $w_-$, we find the desired result.


\appendix

\section{An extension operator from a polygonal meridian domain to a rectangle}
Our aim is to fill the gap in the proof of \cite[Theorem II.3.1]{BDMbook}, that was proved for rectangular $\Omega$ only.

\begin{theorem}
\label{t:BDMext}
If Theorem \ref{t:BDM} holds for any cylindrical domain $\breve\Omega$, i.e., when the meridian domain $\Omega$ is a rectangle, then Theorem \ref{t:BDM} still holds under the more general assumption that $\Omega$ is a Lipschitz polygon such that $\cA\cap\partial\Omega$ has no isolated points.
\end{theorem}

\begin{proof}
Under the above assumptions on its meridian domain, we know that $\breve\Omega$ is a Lipschitz domain. Hence it satisfies the extension property: For any $m\in\N$, there exists an operator $\breve u\mapsto \breve\Pi \breve u$, bounded from $H^m(\breve\Omega)$ to $H^m(\R^3)$ and such that $\breve\Pi \breve u\on{\breve\Omega} = \breve u$. The same result holds if we replace $\R^3$ by a cylinder $\breve\Omega_\cyl \Supset \breve\Omega$. 

We note that the Fourier coefficients $u^k$ of $u$ are the restriction to $\Omega$ of the Fourier coefficients $(\breve\Pi \breve u)^k$ of $\breve\Pi\breve u$.
Then we have the chain of inequalities:
\[
   \DNorm{u^k}{B^m_{(k)}(\Omega)} 
   \le \DNorm{(\breve\Pi\breve u)^k}{B^m_{(k)}(\Omega_\cyl)} 
   \le \beta_{m,\Omega_\cyl} \DNorm{(\breve\Pi\breve u)^k}{H^m_{(k)}(\Omega_\cyl)}
\]
using \eqref{e:normeqB} in the rectangular meridian domain $\Omega_\cyl$ of $\breve\Omega_\cyl$. Starting from any extension operator $\breve\Pi_0: H^m(\breve\Omega)\to H^m(\breve\Omega_\cyl)$ and setting
\[
   \breve\Pi \breve u = \frac{1}{2\pi} \int_0^{2\pi} 
   \cG_{-\theta} \breve\Pi_0 (\cG_{\theta}\breve u)\;\rd\theta
\]
we obtain an extension operator with the same properties as $\breve\Pi_0$, that moreover satisfies the commutation property
\[
   \cG_{\theta}\,\breve\Pi = \breve\Pi\, \cG_{\theta}, \quad \theta\in\T\,.
\]
This implies
\[
   \cF^k\,\breve\Pi = \breve\Pi\, \cF^k,\quad k\in\Z.
\]
By definition
\[
   \DNorm{(\breve\Pi\breve u)^k}{H^m_{(k)}(\Omega_\cyl)} = 
   \DNorm{\cF^k (\breve\Pi \breve u)}{H^m(\breve\Omega_\cyl)}
\]
and we deduce
\[
   \DNorm{\cF^k (\breve\Pi \breve u)}{H^m(\breve\Omega_\cyl)} =
   \DNorm{\breve\Pi (\cF^k \breve u)}{H^m(\breve\Omega_\cyl)} \lesssim
   \DNorm{\cF^k \breve u}{H^m(\breve\Omega)} = 
   \DNorm{u^k}{H^m_{(k)}(\Omega)}
\]
We have proved that
\begin{equation}
\label{e:ineq1}
   \DNorm{u^k}{B^m_{(k)}(\Omega)} \le \beta'_m \DNorm{u^k}{H^m_{(k)}(\Omega)}
\end{equation}
for a constant $\beta'_m$ depending on $m$, $\Omega$ and $\Omega_\cyl$, but not on $k$.

To prove the converse estimate, we will construct (see Lemma \ref{l:exten} below) an extension operator $w\mapsto \Pi w$, bounded from $H^m_1(\Omega)$ to $H^m_1(\Omega_\cyl)$ and such that $\Pi w\on{\Omega} = w$. The same operator will be bounded from $B^m_{(k)}(\Omega)$ to $B^m_{(k)}(\Omega_\cyl)$ for any $k$ with uniformly bounded norms. Then we can deduce:
\[
   \DNorm{u^k}{B^m_{(k)}(\Omega)} \gtrsim 
   \DNorm{\Pi u^k}{B^m_{(k)}(\Omega_\cyl)} \ge
   \beta_{m,\Omega_\cyl}^{-1} \DNorm{\Pi u^k}{H^m_{(k)}(\Omega_\cyl)} 
\]
By definition
\[
   \DNorm{\Pi u^k}{H^m_{(k)}(\Omega_\cyl)} =
   \DNorm{\bx\mapsto(\Pi u^k)(r,z)\,e^{ik\theta}}{H^m(\breve\Omega_\cyl)}
\]
and since $(\Pi u^k)(r,z)\,e^{ik\theta}$ is an extension of $u^k(r,z)\,e^{ik\theta}$:
\[
   \DNorm{\bx\mapsto(\Pi u^k)(r,z)\,e^{ik\theta}}{H^m(\breve\Omega_\cyl)} \ge
   \DNorm{\bx\mapsto u^k(r,z)\,e^{ik\theta}}{H^m(\breve\Omega)} =
   \DNorm{u^k}{H^m_{(k)}(\Omega)}.
\]
We have proved
\begin{equation}
\label{e:ineq2}
   \DNorm{u^k}{B^m_{(k)}(\Omega)} \ge \beta''_m \DNorm{u^k}{H^m_{(k)}(\Omega)}
\end{equation}
for a positive constant $\beta''_m$ depending on $m$, $\Omega$ and $\Omega_\cyl$, but not on $k$.
\end{proof}

It remains to prove the existence of the extension operator $\Pi$:

\begin{lemma}
\label{l:exten}
Let $\Omega$ be a Lipschitz polygon such that $\cA\cap\partial\Omega$ has no isolated points. Let $\Omega_\cyl$ be a rectangle such that $\breve\Omega\Subset\breve\Omega_\cyl$. Let $m\in\N$. 

Then there exists an extension operator $w\mapsto \Pi w$ from $\Omega$ to $\Omega_\cyl$, bounded from $H^m_1(\Omega)$ to $H^m_1(\Omega_\cyl)$, and bounded from $B^m_{(k)}(\Omega)$ to $B^m_{(k)}(\Omega_\cyl)$ for any $k$ with uniformly bounded norms.
\end{lemma}

\begin{proof}
The non-trivial step is to extend functions across conical points, i.e., the corners of $\Omega$ that lie on the axis $\cA$. By localization and partition of unity, we can reduce to the situation where $\Omega$ is a plane sector with one side on $\cA$, the extension being performed to a half-disc $\Omega'$ of same center and same radius. More specifically, choose polar coordinates $(\rho,\varphi)$ in the half-plane $\R_+\times\R$ so that the rotation axis $\cA$ contains the origin and the half-lines $\varphi=0$ and $\varphi=\pi$, namely such that
\[
   z = \rho\,\cos\varphi\quad\mbox{and}\quad r = \rho\,\sin\varphi.
\]
Assume that 
\[
   \Omega = \{(r,z)\in\R_+\times\R, \quad \varphi\in(0,\omega) \;\;
   \mbox{and}\;\; \rho\in(0,R)\}
\]
for some $\omega\in(0,\pi)$ and $R>0$, and set
\[
   \Omega' = \{(r,z)\in\R_+\times\R, \quad \varphi\in(0,\pi) \;\;
   \mbox{and}\;\; \rho\in(0,R)\}.
\]
Choose $\varepsilon=\frac12\min\{\omega,\pi-\omega\}$ and define the sector
\[
   S = \{(r,z)\in\R_+\times\R, \quad 
   \varphi\in(\omega-\varepsilon,\omega+\varepsilon) \;\;
   \mbox{and}\;\; \rho\in(0,R)\}.
\]
Denote
\[
   S_- = S\cap\Omega,\quad S_+ = S\setminus S_-,\quad\mbox{and}\quad
   \Omega_+ = \Omega'\setminus(\Omega\cup S).
\]
Pick $w\in H^m_1(\Omega)$. Since $\rho\simeq r$ in $S$, the restriction of $w$ to $S_-$ lies in the weighted space with non-homogeneous norm, cf \cite{CoDaNi2010},
\[
   J^m_\beta(S_-) = \{v\in L^2_\loc(S_-),\quad 
   \rho^{\beta+m}\partial^\alpha_{(r,z)} v \in L^2(S_-),\;|\alpha|\le m \}
   \quad\mbox{with}\quad \beta = \tfrac12 - m.
\]
By \cite[Theorem 3.23]{CoDaNi2010}, we find that we are in a non-critical case and that $J^m_\beta(S)$ is the direct sum of the space with homogeneous norm
\[
   K^m_\beta(S_-) = \{v\in L^2_\loc(S_-),\quad 
   \rho^{\beta+|\alpha|}\partial^\alpha_{(r,z)} v \in L^2(S_-),\;|\alpha|\le m \}
\]
and the space of polynomials in 2 variables of degree $\le m-2$ denoted by $\P^{m-2}$. Thus $w=w_0+w_1$, with $w_0\in  K^m_\beta(S_-)$ and $w_1\in\P^{m-2}$ with the corresponding estimates.

Using the change of variables $\rho\mapsto t=\log\rho\,$ that transforms $S$ into the strip $\Sigma:=(-\infty,\log R)\times (\omega-\varepsilon,\omega+\varepsilon)$ and $K^m_\beta(S)$ into the space
\[
   \{v\in L^2_\loc(\Sigma),\quad e^{(\beta+1)t} v\in H^m(\Sigma)\}
\]
we can prove the existence of a bounded extension operator $\Pi_0$ from $K^m_\beta(S_-)$ to $K^m_\beta(S)$ such that for all $v\in K^m_\beta(S_-)$, $\Pi_0 v\equiv0$ if $\varphi\ge \omega+\frac{\varepsilon}{2}$. This allows to extend $\Pi_0$ by $0$ on $\Omega_+$, so as to obtain a bounded extension operator from $K^m_\beta(S_-)$ to $H^m_1(S\cup\Omega_+)$. Then we set
\[
   \Pi w = 
   \begin{cases}
   w  &\mbox{in} \quad \Omega\setminus S \\[0.5ex]
   \Pi_0 w_0 + w_1  &\mbox{in} \quad S\cup\Omega_+
   \end{cases}
\]
obtaining a bounded extension operator from $H^m_1(\Omega)$ to $H^m_1(\Omega')$.

Since $K^m_\beta(S)$ is continuously embedded in $H^m_\strong(S)$, we find that $\Pi$ is also bounded from $H^m_\strong(\Omega)$ to $H^m_\strong(\Omega')$. 

Moreover, if a trace $\partial^j_r w\on{\Gamma_0}$ is zero for an index $j\le m-2$, we find that, by construction, with $\Gamma'_0:=\cA\cap\Omega'$,
\[
   \partial^j_r \Pi w\on{\Gamma'_0} = 
   \begin{cases}
   0 & \mbox{in}\quad \Gamma_0 \\
   \partial^j_r w_1 & \mbox{in}\quad \Gamma'_0\setminus\Gamma_0
   \end{cases}
\] 
The trace operator $v \mapsto \partial^j_r v\on{\Gamma'_0}$ is bounded from $H^m_1(\Omega')$ into $H^{m-1-j}(\Gamma'_0)$. Therefore the trace defined as
\[
   \partial^j_r\partial^{m-2-j}_z \Pi w\on{\Gamma'_0} = 
   \begin{cases}
   0 & \mbox{in}\quad \Gamma_0 \\
   \partial^j_r\partial^{m-2-j}_z w_1 & \mbox{in}\quad \Gamma'_0\setminus\Gamma_0
   \end{cases}
\] 
belongs to $H^1(\Gamma'_0)$, hence is continuous across $0$. Since $w_1$ is a polynomial of degree $\le m-2$, $\partial^j_r\partial^{m-2-j}_z w_1$ is a constant. Hence this constant is zero, which proves that $\Pi w$ satisfies on $\Gamma'_0$ the same trace conditions as $w$ on $\Gamma_0$. This ends the proof of the existence of the extension operator $\Pi$ with the required continuity properties.
\end{proof}

\bibliographystyle{siam}
\bibliography{biblio}

\end{document}